\documentclass[12pt,twoside]{amsart}
\usepackage[margin=3cm]{geometry}
\usepackage[colorlinks=false]{hyperref}
\usepackage[english]{babel}
\usepackage{graphicx,titling}
\usepackage{float}
\usepackage{amsmath,amsfonts,amssymb,amsthm}
\usepackage{lipsum}
\usepackage[T1]{fontenc}
\usepackage{fourier}
\usepackage{color}
\usepackage[latin1]{inputenc}
\usepackage{esint}
\usepackage{caption}
\usepackage{ccicons}

\makeatletter
\def\blfootnote{\xdef\@thefnmark{}\@footnotetext}
\makeatother

\newcommand\ccnote{
    \blfootnote{\copyright\,\, William M. Feldman}
    \blfootnote{\ccLogo\, \ccAttribution\,\, Licensed under a \href{https://creativecommons.org/licenses/by/4.0/}{Creative Commons Attribution License (CC-BY)}.}
}

\usepackage[export]{adjustbox}
\numberwithin{equation}{section}
\usepackage{setspace}

\renewcommand{\leq}{\leqslant}

\renewcommand{\geq}{\geqslant}
\renewcommand{\mathbb}{\varmathbb}
\usepackage{fancyhdr}
\newtheorem{theorem}{Theorem}[section]
\newtheorem{lemma}[theorem]{Lemma}
\newtheorem{corollary}[theorem]{Corollary}
\newtheorem{proposition}[theorem]{Proposition}
\newtheorem{remark}[theorem]{Remark}
\usepackage{enumitem}

\newcommand{\eref}[1]{(\ref{e.#1})}
\newcommand{\tref}[1]{Theorem \ref{t.#1}}
\newcommand{\lref}[1]{Lemma \ref{l.#1}}
\newcommand{\pref}[1]{Proposition \ref{p.#1}}
\newcommand{\cref}[1]{Corollary \ref{c.#1}}

\newcommand{\sref}[1]{Section \ref{s.#1}}

\newcommand{\rref}[1]{Remark \ref{r.#1}}
\newcommand{\aref}[1]{Assumption \ref{a.#1}}
\renewcommand{\Box}{\square}

\newcommand{\Z}{\mathbb{Z}}
\newcommand{\R}{\mathbb{R}}

\newcommand{\grad}{\nabla}

\def\Xint#1{\mathchoice
{\XXint\displaystyle\textstyle{#1}}%
{\XXint\textstyle\scriptstyle{#1}}%
{\XXint\scriptstyle\scriptscriptstyle{#1}}%
{\XXint\scriptscriptstyle\scriptscriptstyle{#1}}%
\!\int}
\def\XXint#1#2#3{{\setbox0=\hbox{$#1{#2#3}{\int}$ }
\vcenter{\hbox{$#2#3$ }}\kern-.6\wd0}}

\def\dashint{\Xint-}

\newcommand{\ep}{\varepsilon}

\address{William M Feldman, University of Utah, Department of Mathematics}
\email{feldman@math.utah.edu}

\begin{document}

\thispagestyle{empty}

\begin{minipage}{0.28\textwidth}
\begin{figure}[H]
\includegraphics[width=2.5cm,height=2.5cm,left]{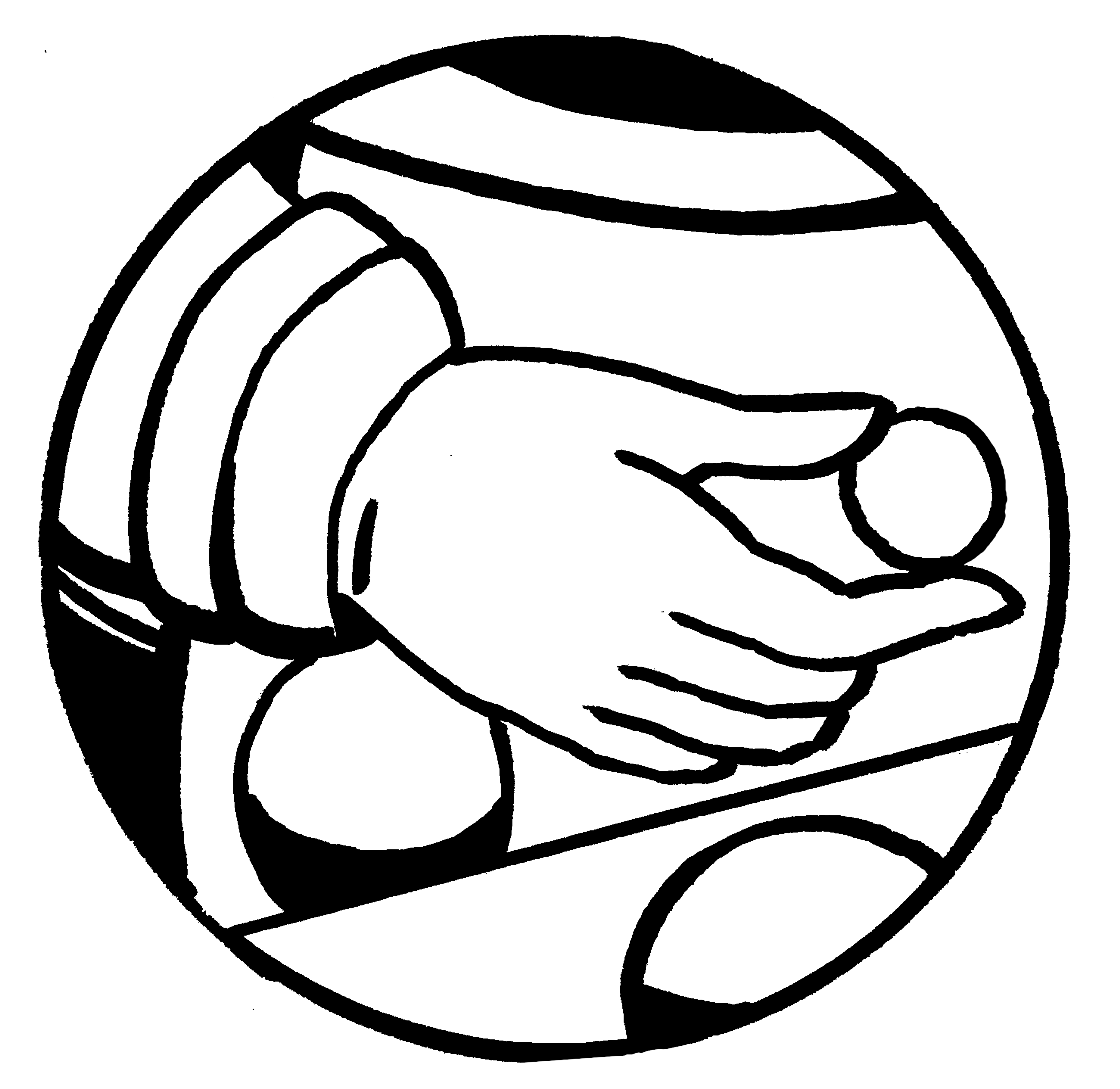}
\end{figure}
\end{minipage}
\begin{minipage}{0.7\textwidth} 
\begin{flushright}
Ars Inveniendi Analytica (2023), Paper No. 2, 41 pp.
\\
DOI 10.15781/52x3-ja93
\\
ISSN: 2769-8505
\end{flushright}
\end{minipage}

\ccnote

\vspace{1cm}


\begin{center}
\begin{huge}
\textit{Large scale regularity of almost minimizers of the one-phase problem in periodic media}

\end{huge}
\end{center}

\vspace{1cm}

\begin{center}
\begin{minipage}[t]{.28\textwidth}
\begin{center}
{\large{\bf{William M Feldman}}} \\
\vskip0.15cm
\footnotesize{University of Utah}
\end{center}
\end{minipage}
\end{center}

\vspace{1cm}

\begin{center}
\noindent \em{Communicated by Scott Armstrong}
\end{center}
\vspace{1cm}


\noindent \textbf{Abstract.} \textit{We prove that minimizers and almost minimizers of one-phase free boundary energy functionals in periodic media satisfy large scale (1) Lipschitz estimates (2) free boundary flat implies Lipschitz estimates.  The proofs are based on techniques introduced by De Silva and Savin~\cite{DeSilvaSavin} for almost minimizers in homogeneous media.}
\vskip0.3cm

\vspace{0.5cm}


\section{Introduction}

We consider minimizers of the one phase free boundary energy functional in a domain $U \subset \R^d$:
\begin{equation}\label{e.energy1}
J(u,U) = \int_U \grad u \cdot a(x)\grad u + Q(x)^2{\bf 1}_{\{u>0\}} \ dx.
\end{equation}
Here the heterogeneous media $Q: \R^d \to (0,\infty)$ and $a: \R^d \to M_{d \times d}^{sym}(\R)$ are $\Z^d$-periodic measurable functions, sufficiently regular (to be specified), satisfying the ellipticity assumptions, for some $\Lambda > 1$,
\begin{equation}\label{e.a-Q-ellipticity-intro}
\Lambda^{-1} I\leq a(x) \leq \Lambda I \ \hbox{ and } \ \Lambda^{-1} \leq Q(x) \leq \Lambda.
\end{equation}
 Consider the domain $U$ to be very large compared to the unit period scale associated with the coefficients.  

The goal of this paper is to prove that almost minimizers of the energy functional $J$ satisfy a Lipschitz estimate and a large scale ``flat implies Lipschitz" estimate for the free boundary $\partial \{u>0\}$.  The results are also new for minimizers, but the generality of almost minimizers is useful for applications.  In a forthcoming paper \cite{FeldmanShape} we will apply these regularity results to obtain quantitative estimates for homogenization rate of a shape optimization problem.  

This follows the now classical Avellaneda and Lin~\cite{AvellanedaLin} idea of inheriting the $C^{1,\alpha}$ regularity iteration from a homogenized problem.  The proof follows the recent quantitative strategy for large scale regularity introduced by Armstrong and Smart \cite{ArmstrongSmart}.  In terms of free boundary regularity theory we are using the very powerful ideas introduced by De Silva and Savin in \cite{DeSilvaSavin}.

Local minimizers of \eref{energy1} satisfy the Euler-Lagrange equation
\begin{equation}\label{e.epeqn}
\begin{cases}
- \grad \cdot(a(x) \grad u) = 0 & \hbox{in } \{u>0\}\cap U\\
n\cdot a(x)n |\grad u|^2 = Q(x)^2 & \hbox{on } \partial \{u>0\} \cap U
\end{cases}
\end{equation}
which is called a one-phase Bernoulli-type free boundary problem. We will study global (almost) energy minimizers which satisfy the additional property that
\[ J(u,B_r) \leq (1+(r/R_0)^\beta)J(v,B_r) \]
for any $B_r \subset U$ and $v \in u+H^1_0(B_r)$.  Here $R_0 \gg 1$ introduces an additional length scale to the problem associated with the almost minimality property, it is a slow scale compared to the unit scale oscillations of the coefficients.  This includes the case of minimizers when $R_0 = +\infty$.

In the large scale limit there is $\Gamma$-convergence to an effective energy functional
\[ J_0(u,U) = \int_U \grad u \cdot \bar{a} \grad u + \langle Q^2 \rangle{\bf 1}_{\{u>0\}} \ dx.\] 
  Here $\langle \cdot \rangle$ is the average over a period cell, and $\bar{a}$ is the effective matrix from classical divergence form elliptic homogenization theory.  The Euler-Lagrange equation associated with critical points of this functional is the classical Bernoulli-type free boundary problem
\begin{equation}\label{e.homeqn}
\begin{cases}
- \grad \cdot (\bar{a} \grad u)= 0 & \hbox{in } \{u>0\}\cap U\\
n\cdot \bar{a}n |\grad u|^2 = \langle Q^2 \rangle& \hbox{on } \partial \{u>0\} \cap U.
\end{cases}
\end{equation}
The free boundary problem \eref{homeqn} has a well developed regularity theory for minimizers and more general solutions of the Euler-Lagrange equation.  In higher dimensions the free boundary may have singular points, but if the free boundary is sufficiently flat in some ball $B_r$ then it will be smooth in $B_{r/2}$.  The theory of almost minimizers of $J_0$ has been developed more recently \cite{DeSilvaSavin,DavidEngelsteinToro,DavidToro}.  We will show that this type of flat implies smooth result for almost minimizers of the effective energy $J_0$ is inherited by almost minimizers of the heterogeneous energy $J$.

The structure of the paper follows the Armstrong and Smart \cite{ArmstrongSmart} quantitative approach to large scale regularity.  First we prove a suboptimal quantitative homogenization result, then we employ the quantitative homogenization result within a $C^{1,\alpha}$ iteration to achieve a Lipschitz estimate of the free boundary.  In order to establish the suboptimal quantitative homogenization result we need to prove some initial regularity of the free boundary.

In this direction, our first main result is a large scale Lipschitz estimate for almost minimizers of $J$.  There are many proofs of the Lipschitz estimate for Bernoulli-type problems, see the book of Velichkov \cite{VelichkovNotes}. Despite some effort the only proof which we were able to adapt to homogenization setting, even for minimizers, is the recent proof of De Silva and Savin \cite{DeSilvaSavin}.  Our argument closely follows theirs with some inputs from homogenization theory.
\begin{theorem}[Large scale Lipschitz estimate]\label{t.main1}
Suppose that $a$ and $Q$ satisfy \eref{a-Q-ellipticity-intro}, $a$ is $\Z^d$-periodic, $1 \leq R \leq c(d,\Lambda,\beta)R_0$, and $u \in H^1(B_R(0))$ satisfies, for all $0 \leq r \leq R$,
\[ J(u,B_r(0)) \leq (1+(r/R_0)^\beta)J(v,B_r(0))\]
 where $v \in u + H^1_0(B_r)$ is the $a$-harmonic replacement.  Then for all $1 \leq r \leq R$
\[ \|\grad u\|_{{\underline{L}}^2(B_r)} \leq C(d,\Lambda)(1+ \|\grad u\|_{\underline{L}^2(B_R)}).\]
If, additionally, $a \in C^{0,\gamma}$ then
\[ |\grad u(0)| \leq C(d,\Lambda,\gamma,[a]_{C^{0,\gamma}})(1+ \|\grad u\|_{\underline{L}^2(B_R)}).\]
\end{theorem}

Note that $Q$ does not need to be periodic for this result, the upper and lower bounds are enough.  Periodicity of $a$ is of course important, interior Lipschitz estimates do not hold for general bounded elliptic coefficient fields.

The other basic regularity estimates of free boundary theory require the Lipschitz estimate and then follow by natural adaptations of established arguments: non-degeneracy (\lref{nondegen}), free boundary perimeter (\lref{bdrystripenergy}) and Hausdorff dimension estimates (\lref{boxcovering}). We also give a proof of inner and outer density estimates in \sref{density}, this is not directly needed for our main result but we include it for usefulness in applications.

With these important intermediate results in hand, we establish our (suboptimal) quantitative homogenization result in \sref{suboptimal-hom}. We show that a regularization $\overline{u}$ of $u$ in the domain $\{u>0\} \cap B_r$ is close to minimizing $J_0$ with the following type of minimality condition
\[J_0(\overline{u},B_{r}) \leq J_0(v,B_r) + C\left[ (r/R_0)^\beta + r^{-\alpha}\right]|B_r| \ \hbox{ for all } \ v \in \overline{u} + H^1_0(B_r).\]  
  Note that there are two error terms in this energy estimate, the first from the almost minimality property of $u$ which is good for $r \ll R_0$, and the second from the homogenization which is good for $r \gg 1$ (the periodicity scale).  Importantly both terms are summable over a geometric sequence of scales between $1$ and $R_0$. This allows us to apply De Silva and Savin's \cite{DeSilvaSavin} $C^{1,\alpha}$ iteration for flat almost minimizers to obtain a flat implies Lipschitz result:
\vfill

{\begin{theorem}[Flat free boundaries are (large scale) Lipschitz graphs]\label{t.main2}
Assume that $a$ and $Q$ satisfy \eref{a-Q-ellipticity-intro}, both are $\Z^d$-periodic, and $a$ is Lipschitz continuous.  Suppose that $u$ is an $(R_0,\beta)$-almost minimizer in $B_R$, $1 \leq R \leq c(d,\Lambda,\beta)R_0$, and $\|\grad u\|_{L^\infty(B_R)} \leq L$.  There is a constant $\delta_0(d,\Lambda,L,\|\grad a\|_\infty)$ so that if
\[ \inf_{\nu \in S^{d-1}}\sup_{x \in B_R} \frac{1}{R}\left|u(x) - \tfrac{\langle Q^2 \rangle^{1/2}}{(\nu \cdot \bar{a} \nu)^{1/2}} (x \cdot \nu)_+\right| \leq \delta_0  \]
then, for all $y \in \partial \{u>0\} \cap B_{R/2}$ and for all $1 \leq r \leq R$
\[ \inf_{\nu \in S^{d-1}}\sup_{x \in B_r(y)} \frac{1}{r}\left|u(x) - \tfrac{\langle Q^2 \rangle^{1/2}}{(\nu \cdot \bar{a} \nu)^{1/2}} ((x-y) \cdot \nu)_+\right| \leq C(d,\Lambda,L,\|\grad a\|_\infty)\left[(r/R)^\alpha \delta_0 + r^{-\omega}\right].\]
And, as a consequence, $\partial \{u>0\} \cap B_{r/2}$ is within distance $C(d,\Lambda,L,\|\grad a\|_\infty)$ of a $1$-Lipschitz graph.
\end{theorem}}

\begin{remark}A large scale flatness assumption is needed because even minimizers of $J_0$ can have singular free boundary points in higher dimensions \cite{JerisonSavin}.
\end{remark}

\begin{remark}Note that both results are new even for minimizers $R_0 = +\infty$.  We consider the more general case of almost minimizers because it adds only minor additional difficulties and it is useful for applications.
\end{remark}
\begin{remark}We must point out that it is extremely important that we are considering energy minimizers or almost minimizers and not general local minimizers or critical points \eref{epeqn}. It is not true in general that solutions of \eref{epeqn}, or even local energy minimizers of $J$, converge to solutions of \eref{homeqn} in the large scale limit \cite{CaffarelliLee,Feldman,FeldmanSmart}.  \end{remark}
\begin{remark}
The matching of the periodicity lattices of $a$ and $Q$ is actually not important at all.  In fact the arguments easily extend to non-periodic $Q$ with a quantitative estimate on the convergence of the spatial averages
\[ \sup_{x} \sup_{r \geq R}|\frac{1}{r^d}\int_{x+[-\frac{r}{2},\frac{r}{2})^d} Q(y)^2 \ dy - \langle Q^2 \rangle | \leq CR^{-\alpha}\]
where the important feature of the H\"older rate is that it is summable over any geometric sequence of scales. Sufficiently mixing stationary ergodic random fields $Q$ where the previous estimate holds without the supremum and with a random variable $C(x)$ satisfying some tail bounds should also be easy to fit into the arguments.

We also expect that analogous results hold for finite range dependence random fields $a(x)$ a la \cite{ArmstrongSmart}. The proofs are set up to be adaptable to this case but this is a more difficult extension and there could be unforeseen challenges.
\end{remark}
\begin{remark}\label{r.liouville-remark}
The result \tref{main2} implies a Liouville result for large scale flat minimizers on the whole space.  This is independently interesting. The Liouville result and related discussion can be found below in \sref{liouville}.  In particular, related to the difficulty of obtaining $C^{1,\alpha}$ estimates of the free boundary all the way down to the unit scale, see \rref{c1alpha-difficulty}. 
\end{remark}
\begin{remark}
The large scale flatness assumption would typically be justified in an application by proving an initial sub-optimal convergence rate to a regular limiting object.  The result \tref{main2} then takes this large scale flatness and iterates it down to all scales above the unit scale.  There is $C^{1,\alpha}$ improvement of flatness at intermediate scales which could be useful in some cases, we did not include the statement here in the introduction but it can be found below in \sref{flatimplieslipschitz}.
\end{remark}
\begin{remark}\label{r.microscopic-a-reg}
Generally it is preferable to remove microscopic regularity hypotheses on the coefficient field, which are philosophically unrelated to large scale regularity.  There are some impediments to completely removing these assumptions in \tref{main2}.  In \tref{main1} we use $a \in C^{0,\gamma}$ to obtain pointwise gradient bounds on $u$.   The pointwise gradient bounds are convenient, but could probably be replaced by the large scale gradient bound in \tref{main1} for most of the elements in the proof of \tref{main2}. 

The key place where we do seem to need local regularity of $u(x)$, and therefore of $a(x)$, is in the strong non-degeneracy estimate \lref{nondegen}.  Strong non-degeneracy can be proved from weak non-degeneracy using a now standard Harnack chain type argument due to Caffarelli \cite{CaffarelliIII} which needs to start from weak non-degeneracy in a, possibly, very small ball near the free boundary, see also \rref{harnack-chain-comment} below. 

Also in the course of proving strong non-degeneracy we use a standard Hopf Lemma barrier construction in \lref{poissonkernel}. This is the only place that Lipschitz as opposed to $C^{0,\gamma}$ is used in order to write the equation in non-divergence form. Hopf Lemma does hold for H\"older continuous coefficients, see  \cite{ApushkinskayaNazarov}, so the Lipschitz assumption could possibly be relaxed to just $C^{0,\gamma}$ with more work.
\end{remark}
\subsection{Literature}\label{s.literature}

The idea that elliptic problems in heterogeneous media could inherit regularity, up to a certain threshold, from their large scale effective limit was first introduced by Avellaneda and Lin in a series of papers \cite{AvellanedaLin}.  Often this kind of regularity theory can be used to obtain optimal quantitative convergence estimates for the homogenization limit.  This idea was re-emphasized by Armstrong and Smart  \cite{ArmstrongSmart} who were able to adapt the Avellaneda and Lin idea for the first time in random media.  The idea of \cite{ArmstrongSmart} is to first obtain a suboptimal quantitative homogenization result, which can then be used within the $C^{1,\alpha}$ iteration as long as the quantitative rate is summable over geometric sequences of scales.  It is exactly this idea which we are applying in the context of free boundary problems.

The regularity theory for minimizers of $J_0$ was first significantly developed by Alt and Caffarelli \cite{AltCaffarelli}.  The topic has been extremely popular since then and we do not make an attempt at a thorough recounting.   We will just point out the most relevant, which is the recent development of regularity theory of almost minimizers.  A series of papers by (various subsets of) the authors David, Engelstein, Smit Vega Garcia, and Toro \cite{DavidToro,DavidEngelsteinToro,DavidEngelsteinSVGToro} have initiated and made major developments in the study of almost minimizers of the more general two-phase version of the energy $J$.  David, Engelstein, Smit Vega Garcia and Toro \cite{DavidEngelsteinSVGToro} (two-phase) and, independently, Trey \cite{Trey} (vector-valued) have proved a Lipschitz estimate for almost minimizers of variable coefficient Bernoulli-type problems. De Silva and Savin \cite{DeSilvaSavin} gave another account of the regularity theory for almost minimizers of the one phase problem, and, as mentioned before, we will be following their ideas.  Note that all these works are considering small scale regularity, while we are considering large scales.  In other words one can consider our work to be on oscillatory coefficients $a(x/\ep)$ which do not satisfy a uniform in $\ep>0$ H\"older regularity hypothesis.

There are very few works on the large scale regularity theory of free boundary problems in heterogeneous media.  The only other result we are aware of is by Aleksanyan and Kuusi \cite{AleksanyanKuusi}. They obtain a result analogous to ours for a version of the obstacle problem in random media.    A major difficulty in studying the Bernoulli free boundary problem, as far as we understand it, is that the Euler-Lagrange equation \eref{epeqn} cannot be enough information by itself to get regularity.  This is already reflected in techniques for the homogeneous problem which variously use the additional properties of being a minimizer, almost minimizer, maximal subsolution, or minimal supersolution.  As established in \cite{CaffarelliLee,Kim,FeldmanSmart,Feldman} the large scale limit of minimal supersolutions, maximal subsolutions, and minimizers are all rather different Bernoulli free boundary problems and thus regularity theory techniques would need to be adapted to each case.  In fact the regularity theory for the homogenized problem for minimal supersolutions and maximal subsolutions is not understood yet \cite{Feldman}.  Thus we were led to consider first the case of minimizers and almost minimizers as we do in this paper.  In this case the large scale effective problem is a standard constant coefficient minimization problem and the regularity theory is well understood.
\subsection{Acknowledgments} The author was supported by the NSF grant DMS-2009286.  The author would like to thank Farhan Abedin for extensive and motivating discussions on the topics of the paper.

\section{Background results}
This section introduces notations and conventions which will be used in the paper and review some well known results from the literature about divergence form elliptic homogenization.  \subsection{Assumptions and conventions} As described in the introduction we will study energy functionals of the following type
\begin{equation*}
J(u,U) = \int_U \grad u \cdot a(x)\grad u + Q(x)^2{\bf 1}_{\{u>0\}} \ dx.
\end{equation*}

We make precise the assumptions on the coefficients which will be in place throughout the paper, unless otherwise specified in some few locations.

\begin{enumerate}[label = (Q\arabic*)]
\item \label{a.Q1} The coefficient $Q: \R^d \to (0,\infty)$ is assumed to satisfy the ellipticity condition
\begin{equation*}\label{e.Qellipticity}
\Lambda^{-1} \leq Q(x) \leq \Lambda.
\end{equation*}
\item \label{a.Q2} The coefficient $Q$ is assumed to be $\Z^d$-periodic and satisfy the normalization condition
\begin{equation*}\label{e.Qnormalization}
\langle Q^2 \rangle = \int_{[0,1]^2} Q(x)^2 \ dx = 1
\end{equation*}
\end{enumerate}
\begin{enumerate}[label = (a\arabic*)]
\item \label{a.a1} The coefficients $a: \R^d \to M_{d \times d}^{sym}(\R)$ are $\Z^d$-periodic and satisfy the ellipticity bounds
\begin{equation*}\label{e.aellipticity}
\Lambda^{-1} I\leq a(x) \leq \Lambda I 
\end{equation*}
\item \label{a.a2}The coefficient field $a: \R^d \to M_{d \times d}^{sym}(\R)$ is Lipschitz.
\end{enumerate}
See \rref{microscopic-a-reg} above for comments about the role of the microscopic regularity \aref{a1}.
\begin{enumerate}[label=$\bullet$]
\item  Constants $C$ and $c$ in the text below which depend at most on $d$, $\Lambda$, $\beta$, $\|\grad a\|_\infty$ are called \emph{universal}.  Such constants may change meaning from line to line without mention.  If we intend to fix the value of such a universal constant for a segment of the argument we may denote it as $c_0$, $C_0$, $c_1$ etc.
\end{enumerate}

Notations:
\begin{enumerate}[label = $\bullet$]
\item Balls are denoted $B_r$ (if the center does not need denotation) or $B_r(x)$. Boxes are denoted
\[ \Box_r = [-r/2,r/2)^d \ \hbox{ and } \ \Box_r(x) = x + \Box_r.\]
\item Lebesgue measure of a measurable set $U$ is denoted $|U|$.  This notation may also be used for lower dimensional measures if the set is obviously a particular Hausdorff dimension, i.e. $|\partial B_r|$ is the surface area of the sphere of radius $r$.
\item Averaged integrals \[ \dashint_U f\  dx = \frac{1}{|U|} \int_U f \ dx\]
\item Averaged $L^p$ norms \[\|f\|_{{\underline{L}}^p(U)} = (\frac{1}{|U|} \int_U |f|^p \ dx)^{1/p}.\]
\end{enumerate}

\subsection{Divergence form elliptic equations with periodic coefficients}\label{s.homrecall}
We consider elliptic energy functionals of the type
\[ E(u,U) = \int_{U}   \grad u\cdot a(x)\grad u \ dx\]
evaluated on a domain $U$ of $\R^d$. 

Given $g \in W^{1,p}(U)$ for some $p \geq 2$ we can consider the Dirichlet problem
\[ \min \{ E(u,U) : u \in g + H^1_0(U)\}\]
with the associated Euler-Lagrange PDE
\begin{equation}
\begin{cases}
-\grad \cdot (a(x) \grad u) = 0 & \hbox{in } U\\
u =g & \hbox{on } \partial U
\end{cases}
\end{equation}
all interpreted in the appropriate weak sense.

We introduce the correctors $\chi_q$ which are the unique global $\Z^d$ periodic and mean zero solutions of the problem
\[ - \grad \cdot (a(x)(q+\grad \chi_q)) = 0  \ \hbox{ in } \ \R^d\]
existing for each $q \in \R^d$.  It follows from the previous properties that $\chi_q$ depends linearly on $q \in \R^d$.  

We make a few notes about corrector bounds which will be useful. By multiplying the corrector equation by $q$ and integrating over a unit period cell,
\begin{equation*}\label{e.correctorest}
\|\grad \chi_q\|_{L^2([0,1)^d)}\leq \Lambda^2|q|.
\end{equation*}
By Poincar\'e and mean zero also
\begin{equation*}\label{e.correctorest2}
\|\chi_q\|_{L^2([0,1)^d)}\leq C|q|
\end{equation*}
and using periodicity we can also conclude that for all $r \geq 1$
\begin{equation}\label{e.correctorest3}
\|\grad \chi_q\|_{\underline{L}^2(B_r)}+\|\chi_q\|_{\underline{L}^2(B_r)}\leq C|q|
\end{equation}
for a universal $C \geq 1$.  By standard elliptic regularity theory \eref{correctorest3} can be upgraded to an $L^\infty$ estimate if we assume that $a \in C^{0,\gamma}$, the new constant $C$ will also depend on $\gamma$ and $[a]_{C^{0,\gamma}}$.

The homogenized matrix is defined
\begin{equation}\label{e.baradef}
 \bar{a}_{ij} =\langle (e_i + \grad \chi_{e_i}(x))a(x)(e_j + \grad \chi_{e_j}(x)) \rangle 
 \end{equation}
where $e_i$ are the standard basis vectors.  We define the homogenized energy functional
\[E_0(u,U) = \int_{U}   \grad u\cdot \overline{a}\grad u \ dx \]

Given these definitions the following convergence result is classical.  See \cite{ArmstrongSmart,AKM} for a proof in the more difficult context of random media.
\begin{theorem}[Homogenized replacement in balls]\label{t.quanthomregulardomain}
Suppose $r \geq 1$, $g \in W^{1,p}(B_r)$ for some $p > 2$, and $u$ and $u_0$ respectively minimize $E$ and $E_0$ over $H^1_0(B_r) + g$. Then there are $\alpha \in (0,1)$ and $C \geq 1$ depending on $d$, $\Lambda$, and $p$ so that
\[ |E(u,B_r) - E_0(u_0,B_r)| +\frac{1}{r^2}\| u -  u_0\|^2_{L^2(B_r)} \leq C\|\grad g\|_{\underline{L}^p(U)}^2r^{-\alpha}|B_r|.\]
\end{theorem}
This result does hold in the more general class of Lipschitz domains, and that assumption can be pushed a bit further, but we will not use that generality in our paper.  We will use a slightly different homogenization result in the positivity set $\{u>0\}$ which takes advantage of the a-priori Lipschitz bound from \sref{Lipschitz}, see \sref{recallquanthom} for that statement and more discussion.

From an initial quantitative homogenization result one can also prove a version of elliptic regularity for $a$-harmonic functions.  We will make use of the following facts.

\begin{theorem}[Interior regularity]\label{t.hominteriorreg}

(1) (Lipschitz estimate \cite[Lemma 16]{AvellanedaLin}) If $u$ is $a$-harmonic in $B_r$ then
\[ \|\grad u\|_{L^\infty(B_{r/2})} \leq \frac{C}{r}\|u\|_{L^\infty(B_r)}\]
where the constant $C$ is universal.

\noindent (2) ($C^{1,\alpha}$ estimate \cite[Lemma 15]{AvellanedaLin}) If $u$ is $a$-harmonic in $B_R$ then for any $r_0 \leq r \leq R$
\[  \sup_{B_{r}}|u(x) - (A+q \cdot x + \chi_q(x))| \leq C(\frac{r}{R})^{1+\alpha} \|u\|_{\underline{L}^2(B_R)}\]
for some $q$ and $A$ depending on $r$ with $|A| + r|q| \leq C\|u\|_{\underline{L}^2(B_R)}$.  And, since $A + q \cdot x + \chi_q$ is $a$-harmonic,
\[ \|\grad u - (q + \grad \chi_q)\|_{\underline{L}^2(B_{r})} \leq C(\frac{r}{R})^{\alpha}\frac{1}{R}\|u\|_{\underline{L}^2(B_R)}\,.\]
\end{theorem}
An analogous $C^{1,1}$ estimate is also true, the general theory in \cite{AKM} explains this, but the statements there are for random media and in a less convenient form for our purposes.

\section{Lipschitz estimate}\label{s.Lipschitz}

In this section we prove the Lipschitz estimate \tref{main1} for almost minimizers of energy functionals of the form
\[ J(u,U) = \int_U a(x) \grad u \cdot \grad u + Q(x)^2 {\bf 1}_{\{u>0\}} \ dx.\]
In this section $a(x)$ will be $\Z^d$-periodic, and $a$ and $Q$ will satisfy the ellipticity bounds i.e. we are under \aref{a1} and \aref{Q1}.  We do not need to assume any further structural assumptions (e.g. periodicity, regularity) on $Q$ at this stage.

A competitor $u \in H^1(U)$ is called a $(R_0,\beta)$-almost minimizer in a domain $U$ if, for any $B_r \subset U$, and any $v \in u + H^1_0(B_r)$
\begin{equation}\label{e.almostminimal}
 J(u,B_r) \leq (1+(r/R_0)^\beta)J(v,B_r) .
 \end{equation}
 \begin{remark}
 As mentioned in the statement of \tref{main1}, for the purposes of the Lipschitz estimate we only need the almost minimizing property to hold for the $a$-harmonic replacement.  That fact is convenient for making translations between slightly different notions of almost minimality.
 
We also remark that the proofs go through with almost no change with an additional additive term: for any $B_r \subset U$, and any $v \in u + H^1_0(B_r)$
\begin{equation}\label{e.almostminimaladditive}
 J(u,B_r) \leq (1+(r/R_0)^\beta)J(v,B_r) +(r/R_0)^\beta|B_r|.
 \end{equation}
 \end{remark}
 
We follow the approach of De Silva and Savin \cite{DeSilvaSavin} which proceeds in three steps:
\begin{itemize}
\item Step 1 (Dichotomy): if an almost minimizer $u$ has large $L^2$-averaged slope in a ball $B_r$ then either the $L^2$ averaged slope decreases by half in a smaller ball $B_{\eta r}$ or  $u$ is close to a corrected linear $a$-harmonic function in $B_{\eta r}$.
\item Step 2 (Interior-like improvement): In the second case of the previous dichotomy $u$ has interior $C^{1,\alpha}$-like improvement down to the unit scale.  
\item Step 3 (Iteration): Iterating the dichotomy either case results in a Lipschitz bound.
\end{itemize}
The main point for us is to use the correctors and the Avellaneda-Lin \cite{AvellanedaLin} interior estimates \tref{hominteriorreg} in place of the corresponding results for homogeneous problems.  Although this sounds rather straightforward, we were unable to make such an idea succeed in any of the other standard proofs of the Lipschitz estimate from the literature, see \cite{VelichkovNotes} for a survey of such methods.

\subsection{Step 1: Dichotomy}
\begin{lemma}\label{l.DSLIPP1}
Let $u \in H^1(B_r)$ and suppose that
\[ J(u,B_r) \leq (1+\sigma)J(v,B_r) \ \hbox{ for all } \ v \in u + H^1_0(B_r).\]
 Let $\frac{1}{4} \geq \delta>0$ there exist constants $\eta$, $M$ and $\sigma_0$ (depending on $\delta$ and universal constants) so that if $\sigma \leq \sigma_0$, $\eta r \geq 1$, and 
\[ \|\grad u\|_{{\underline{L}}^2(B_r)}  \geq M\]
then the following dichotomy holds: either
\[ \|\grad u\|_{{\underline{L}}^2(B_{\eta r})} \leq \frac{1}{2} \|\grad u\|_{{\underline{L}}^2(B_r)} \]
or 
\[\|\grad u - (q + \grad \chi_q(x))\|_{{\underline{L}}^2(B_{\eta r})}  \leq \delta \|\grad u\|_{{\underline{L}}^2(B_r)} \]
with some $q \in \R^n$ satisfying
\[ C_0^{-1}\|\grad u\|_{{\underline{L}}^2(B_r)} \leq |q| \leq C_0 \|\grad u\|_{{\underline{L}}^2(B_r)}\]
with a constant $C_0 \geq 1$ universal (not depending on $\delta$).
\end{lemma}
\begin{proof}
1. Let $v$ be the $a$-harmonic replacement of $u$ in $B_1$.  Then
\[ \int_{B_r}  a(x) \grad (u-v) \cdot \grad (u-v) \ dx\leq J(u,B_r)+\int_{B_r} a(x) \grad v \cdot \grad v - 2 a(x)\grad u \cdot \grad v \ dx.\]
Using the minimizing property of $u$ and ellipticity
\[ \Lambda^{-1}\|\grad u - \grad v\|^2_{L^2(B_r)} \leq (1+\sigma)J(v,B_r) + \int_{B_r} a(x) \grad v \cdot \grad v - 2 a(x)\grad u \cdot \grad v \ dx.\]
Then combining the $|\grad v|^2$ terms on the right
\[ \Lambda^{-1} \|\grad u - \grad v\|^2_{L^2(B_r)} \leq \sigma \|\grad v\|^2_{L^2(B_r)} + \Lambda |B_r| + \int_{B_r}  2a(x) \grad v \cdot \grad (v-u) \ dx.\]
The last term is $0$ because $v$ is $a$-harmonic so
\[ \|\grad u - \grad v\|^2_{\underline{L}^2(B_r)} \leq C \sigma \|\grad v\|^2_{\underline{L}^2(B_r)} + C \leq C (\sigma \|\grad u\|^2_{\underline{L}^2(B_r)} + 1).\]

2. Since $v$ is $a$-harmonic in $B_r$ we can apply the interior $C^{1,\alpha}$ estimate \tref{hominteriorreg} to find that, as long as $\eta r \geq 1$,
\[ \|\grad v - (q + \grad \chi_q)\|_{\underline{L}^2(B_{\eta r})} \leq C\eta^\alpha \|\grad v\|_{\underline{L}^2(B_r)} \leq C \eta^\alpha \|\grad u\|_{\underline{L}^2(B_r)}\]
for some
\[ |q| \leq  C\|\grad u\|_{\underline{L}^2(B_r)}.\]
Note also, by the corrector estimates \eref{correctorest3},
\[ \|\grad \chi_q\|_{\underline{L}^2(B_r)} \leq C|q|.\]

3. Now we combine the previous parts to estimate
\begin{align*}
 \|\grad u - (q + \grad \chi_q(x))\|_{\underline{L}^2(B_{\eta r})}&\leq \|\grad v - (q + \grad \chi_q(x))\|_{\underline{L}^2(B_{\eta r})} + \|\grad u - \grad v\|_{\underline{L}^2(B_{\eta r})}\\
&\leq C \eta^\alpha \|\grad u\|_{\underline{L}^2(B_r)}+C \sigma^{1/2}\eta^{-d/2} \|\grad u\|_{\underline{L}^2(B_r)} + C\eta^{-d/2}
\end{align*}
and so, also,
\[ \|\grad u\|_{\underline{L}^2(B_{\eta r})} \leq C \eta^\alpha \|\grad u\|_{\underline{L}^2(B_r)}+C \sigma^{1/2}\eta^{-d/2} \|\grad u\|_{\underline{L}^2(B_r)} + C\eta^{-d/2} + C|q|. \]
Both estimates require that $\eta r \geq 1$.

Now, we first fix $\eta>0$ small depending on $\delta$, and then fix $\sigma_0>0$ small and $M>1$ large depending on $\delta$ and the choice of $\eta$ so that
\begin{align*}
 C \eta^\alpha \|\grad u\|_{\underline{L}^2(B_r)}+&C \sigma^{1/2}\eta^{-d/2} \|\grad u\|_{\underline{L}^2(B_r)} + C\eta^{-d/2} \\
    &\leq \frac{1}{4} \delta \|\grad u\|_{\underline{L}^2(B_r)} + C\eta^{-d/2}M^{-1}\|\grad u\|_{\underline{L}^2(B_r)} \\
    &\leq  \delta \|\grad u\|_{\underline{L}^2(B_r)}.
 \end{align*}
 Now, as long as $\eta r \geq 1$, we have the two inequalities
 \[  \|\grad u - (q + \grad \chi_q(x))\|_{\underline{L}^2(B_{\eta r})}\leq  \delta \|\grad u\|_{\underline{L}^2(B_r)}\]
 and, using $\delta \leq \frac{1}{4}$,
 \[  \|\grad u\|_{\underline{L}^2(B_{\eta r})} \leq  \frac{1}{4} \|\grad u\|_{\underline{L}^2(B_r)} + C_1|q|. \]
 
 4. Finally we divide up into the dichotomy in the statement.  If 
 \[ |q| \leq \frac{1}{4}C_1^{-1}\|\grad u\|_{\underline{L}^2(B_r)}\]
 then
 \[ \|\grad u\|_{\underline{L}^2(B_{\eta r})} \leq  \frac{1}{2} \|\grad u\|_{\underline{L}^2(B_r)}\]
 while, otherwise, 
 \[ \|\grad u - (q + \grad \chi_q(x))\|_{\underline{L}^2(B_{\eta r})}\leq  \delta \|\grad u\|_{\underline{L}^2(B_r)}\]
 with some
 \[ \frac{1}{4}C_1^{-1}\|\grad u\|_{\underline{L}^2(B_r)} \leq |q| \leq C\|\grad u\|_{\underline{L}^2(B_r)}.\]
\end{proof}

\subsection{Step 2: Close to planar implies interior-like $C^{1,\alpha}$ improvement} The next lemma says that if $\grad u$ is $\delta|q|$ close to a corrector gradient $q + \grad \chi_q$ in $\underline{L}^2$, then there is an interior $C^{1,\alpha}$-like improvement of oscillation.  Basically an energy argument shows that the zero level $\{u = 0\}$ must be relatively small and so the oscillation improvement of the $a$-harmonic replacement carries over to $u$.

\begin{lemma}\label{l.DSLIPP2}
Let $u$ as in \lref{DSLIPP1} and $r \geq 1$.  Suppose that $|q| \geq1$ and
\[ \|\grad u - (q + \grad \chi_q(x))\|_{\underline{L}^2(B_{r})} \leq \delta |q| \]
for some $1>\delta>0$.

  There are $\alpha_0>0$, $\mu>0$, $1>\bar{\delta}>0$ and $c_0>0$ universal such that if
\[ \delta \leq \bar{\delta}, \ \sigma \leq c_0  \delta^2, \ \hbox{ and } \ \mu r \geq 1\]
then
\[ \|\grad u - (q' + \grad \chi_{q'}(x))\|_{\underline{L}^2(B_{\mu r})} \leq \mu^{\alpha_0}\delta   |q'| \]
with $q' \in \R^d$ satisfying
\[ |q' - q| \leq C_1 \delta |q|\]
for a universal $C_1\geq 1$.
\end{lemma}
The result should be possible with any $\alpha \in (0,1)$ with parameters depending on $\alpha$ using the $C^{1,1}$ version of \tref{hominteriorreg}.

\begin{proof}
1. (Harmonic replacement) Let $\bar{v}$ be the $a$-harmonic replacement of $u$ in $B_{r/2}$ and let $v$ be
\[ v = u \ \hbox{ in } \ B_r \setminus B_{r/2} \ \hbox{ and }\ v = \bar{v} \ \hbox{ in } \ B_{r/2}.\]
Then
\[ J(u,B_r) \leq (1+\sigma)J(v,B_r)\]
or
\[ J(u,B_{r/2}) \leq \sigma J(u,B_r \setminus B_{r/2}) + (1+\sigma) J(v,B_{r/2}).\]
Also note, using the corrector estimate $\| \grad \chi_q\|_{\underline{L}^2(B_r)} \leq C|q|$ from \eref{correctorest3},
\[ \|\grad u\|_{\underline{L}^2(B_{r})} \leq C|q|.\]
  Thus
\[ \int_{B_{r/2}} a(x) \grad u \cdot \grad u - a(x) \grad v \cdot \grad v \ dx +\int_{B_{r/2}} Q(x) {\bf 1}_{\{u>0\}} \ dx \leq \sigma (C|q|^2+\Lambda)|B_r| + \int_{B_{r/2}} Q(x) \ dx.\]
Since $v$ is the $a$-harmonic replacement for $u$ in $B_{r/2}$
\[ \int_{B_{r/2}} a(x) \grad u \cdot \grad u - a(x) \grad v \cdot \grad v \ dx = \int_{B_{r/2}} a(x) \grad (u-v) \cdot \grad (u-v) \ dx\]
so, plugging that in, re-arranging, and using ellipticity,
\[ \|\grad u - \grad v\|_{\underline{L}^2(B_{r/2})}^2 \leq C\sigma (|q|^2 + 1) + \frac{1}{|B_{r/2}|} \int_{B_{r/2}} Q(x){\bf 1}_{\{ u = 0\}} \ dx\]
or
\begin{equation}\label{e.uvzerosetest}
 \|\grad u - \grad v\|_{\underline{L}^2(B_{r/2})}^2 \leq C\sigma |q|^2 + \Lambda \frac{|\{ u(x) =0\} \cap B_{r/2}|}{|B_{r/2}|}
 \end{equation}

2. (Large slope implies small measure of zero level) Next we claim that
\[ \frac{|\{ u(x) =0\} \cap B_{r/2}|}{|B_{r/2}|} \leq C\delta^{2+\beta} \]
for some universal $C,\beta>0$.  Actually it will hold for $2+\beta \leq 2^* = \frac{2d}{d-2} = 2 + \frac{4}{d-2}$ the Sobolev embedding exponent for $L^2$ (or $2+\beta < 2^* = +\infty$ in $d=2$).

Note that
\[ \int_{B_r} (u - \ell) \ dx = 0 \ \hbox{ where } \ \ell(x) = q \cdot x + \chi_q(x) + \dashint_{B_r} u - \chi_q \ dx.\]
So, by Poincar\'e,
\[ c\|u - \ell\|_{\underline{L}^2(B_{r})} \leq  r\|\grad u - \grad \ell\|_{\underline{L}^2(B_{r})} \leq \delta |q| r.\]
Since $u \geq 0$ 
\[ c\|(\ell)_-\|_{\underline{L}^2(B_{r})}\leq \delta |q| r.\]
If $q \cdot x + \dashint_{B_r} (u - \chi_q )$ were not positive in $B_{3r/4}$ then $|q \cdot x| \geq \frac{1}{8}|q|r$ in some ball of radius $r/16$ contained in $B_r$ so, using $\|\chi_q\|_{\underline{L}^2(B_r)} \leq K|q|$ from \eref{correctorest3},
\[ \|(\ell)_-\|_{\underline{L}^2(B_{r})} \geq c|q|r - \|\chi_q\|_{\underline{L}^2(B_{r})} \geq c|q| (r-c^{-1}K)\]
for a universal constant $c$.  This is a contradiction for $r \geq 2c^{-1}K$ and $\delta$ sufficiently small depending on $K$ and other universal constants.  Note that, given the hypothesis $\mu r \geq 1$, we can think of this as a requirement on the choice of $\mu$, that we must choose $\mu \leq \frac{1}{2}cK^{-1}$.

Thus
\[ \ell(x) \geq  \frac{1}{4}|q|r - \|\chi_q\|_{\infty} \geq \frac{1}{4}|q|(r-4K) \geq c|q|r \ \hbox{ in } \ B_{r/2}\]
for a universal $c>0$ and using $r \geq 8K$ (imposing additionally $\mu \leq \frac{1}{8K}$).

By Poincar\'e-Sobolev
\[ \|u - \ell\|_{\underline{L}^{2^*}(B_{r})} \leq Cr\|\grad u - \grad \ell\|_{\underline{L}^2(B_{r})} \leq C\delta |q| r \]
but also
\[ \|u - \ell\|_{\underline{L}^{2^*}(B_{r})} \geq c|q| r \frac{|\{ u = 0\} \cap B_{r/2}|^{1/2^*}}{|B_{r/2}|^{1/2^*}}.\]
So the claim follows combining the previous two inequalities.

3. Note that $v(x) - (q \cdot x + \chi_q(x))$ is the $a$-harmonic replacement of $u(x)- (q \cdot x + \chi_q(x))$.  So
\[ \|\grad v - (q + \grad \chi_q)\|_{\underline{L}^2(B_{r/2})} \leq \Lambda^2\|\grad u - (q + \grad \chi_q)\|_{\underline{L}^2(B_{r/2})} \leq \delta \Lambda^2 |q|.\]
By the Avellaneda-Lin interior $C^{1,\alpha}$ estimates for $a$-harmonic functions \tref{hominteriorreg} applied to $v(x) - (q \cdot x + \chi_q(x))$, for any $ 0 < \mu < 1/2$ such that $\mu r \geq 1$,
\[\|\grad v - (q + \grad \chi_q) - (\bar{q}+ \grad \chi_{\bar{q}})\|_{\underline{L}^2(B_{\mu r})} \leq C\mu^\alpha \|\grad v - (q + \grad \chi_q)\|_{\underline{L}^2(B_{r/2})}  \leq C \mu^\alpha \delta |q|\]
for some $\bar{q} \in \R^d$ with
\[ |\bar{q}| \leq C\|\grad v - (q + \grad \chi_q)\|_{\underline{L}^2(B_{r/2})} \leq C\delta |q|.\]
Write $q' = q + \bar{q}$ and recall the linearity of the corrector $\chi_q + \chi_{\bar{q}} = \chi_{q + \bar{q}}$.  Assuming $C\delta \leq \frac{1}{2}$ we have $ 2|q'| \geq |q| \geq 1$.

Combining this with estimate \eref{uvzerosetest}
\[ \|\grad u - (q' + \grad \chi_{q'})\|^2_{\underline{L}^2(B_{\mu r})} \leq C\sigma \mu^{-d}|q|^2 +C \mu^{-d}\delta^{2+\beta} + C\mu^{2\alpha} \delta^2 |q|^2.\]
Now choose parameters in the following order: let $\alpha_0 = \alpha/2$ (or anything smaller than $\alpha$ works), then choose $\mu$ small enough so that $C\mu^{2(\alpha-\alpha_0)} \leq \frac{1}{6}$ for the third term (in addition to the previous requirements), then choose $\delta \leq \bar{\delta}$ with $\bar{\delta}^{\beta} = \frac{1}{6}C^{-1}\mu^{d+2\alpha_0}$ for the second term.  Finally the condition $\sigma \leq \frac{1}{6}C^{-1}\delta^2\mu^{d+2\alpha_0}$ gives that all three terms above are smaller than $\frac{1}{6}\mu^{2\alpha_0}\delta^2 |q'|^2$ (note we are using $|q'| \geq \frac{1}{2}|q| \geq \frac{1}{2}$).

\end{proof}

\begin{lemma}\label{l.DSinteriorLIP}
Suppose that $u$ is a $(R_0,\beta)$-almost minimizer of $J$ in $B_R$ for some $R_0 \geq R \geq 1$ and for some $|q| \geq 2$
\[ \|\grad u - (q + \grad \chi_q(x))\|_{\underline{L}^2(B_{R})} \leq \delta |q|. \]
There is universal $\delta_*>0$ so that if $\delta \leq \delta_*$ then
\[  \|\grad u\|_{\underline{L}^2(B_{r})} \leq C|q| \ \hbox{ for all } \ 1 \leq r \leq R.\]
\end{lemma}

\begin{proof}
We iterate \lref{DSLIPP2} with $\alpha = \beta/2$.  Pick $\delta_*$ so that
\[ \delta_* \leq \bar{\delta} \ \hbox{ and } \ 4C_1\frac{1}{1-\mu^\alpha}\delta_* \leq 1 \]
where $\bar{\delta}$ is the universal constant from \lref{DSLIPP2}.

It suffices to consider the case
\begin{equation}\label{e.initialsigmaineq}
\sigma_0 := (R/R_0)^{\beta} \leq c_0 \delta_*^2
\end{equation}
 where $c_0$ is the universal constant from \lref{DSLIPP2} and $\delta_*$ is defined above.

   We claim that for all $j$ such that $\mu^jR \geq 1$ there are $q_j$ such that
\begin{equation}\label{e.induction1}
\|\grad u - (q_j + \grad \chi_{q_j}(x))\|_{\underline{L}^2(B_{\mu^j R})} \leq \delta_* \mu^{\alpha j}   |q_j|
\end{equation}
with
\begin{equation}\label{e.induction2}
 |q_{j+1} - q_j| \leq C_1\delta_* \mu^{\alpha j} |q_j| \ \hbox{ and } \ 1 \leq |q_j| \leq 2|q|.
 \end{equation}
 
 We will prove this by induction.  By the hypothesis of the theorem \eref{induction1} and \eref{induction2} hold for $j=0$.  Suppose that \eref{induction1} and \eref{induction2} hold for all $0 \leq j \leq k$.  

Note that we have
\[ |q_{k+1}-q| \leq \sum_{j=0}^{k}|q_{j+1}-q_j| \leq  2C_1\frac{1}{1-\mu^\alpha}\delta_* |q|.\]
In particular 
\[ |q_{k+1}| \leq (1+2C_1\frac{1}{1-\mu^\alpha}\delta_*)|q| \leq 2|q|\]
because of the choice of $\delta_*$, and, similarly,
\[ |q_{k+1}| \geq (1-2C_1\frac{1}{1-\mu^\alpha}\delta_*)|q| \geq \frac{1}{2}|q| \geq 1.\]
Thus we have established the second part of \eref{induction2} for $k+1$.

If $ \mu^{k+1}R \leq 1$ then we are already done.  Otherwise $u$ is in the set up of \lref{DSLIPP2} in $B_{\mu^kR}$ with \eref{induction1} and
\begin{equation}\label{e.hypotheses-check-1}
 \delta_k = \delta_* \mu^{\alpha k} \ \hbox{ and } \ \sigma_k =  (r_k/R_0)^\beta =  \mu^{k\beta} \sigma_0.
 \end{equation}
Note that $\delta_k \leq \delta_* \leq \bar{\delta}$ and
\begin{equation}\label{e.hypotheses-check-2}
 \sigma_k = \mu^{2k\alpha} \sigma_0 = \delta_k^2\delta_*^{-2}\sigma_0 \leq c_0\delta_k^2
 \end{equation}
by \eref{initialsigmaineq}. Thus all the assumptions of \lref{DSLIPP2} hold and we find that there is $q_{k+1}$ such that
\[ \|\grad u - (q_{k+1} + \grad \chi_{q_{k+1}}(x))\|_{\underline{L}^2(B_{\mu^{k+1}R})} \leq \delta \mu^{\alpha(k+1)}   |q_{k+1}|\]
and with
\[ |q_{k+1} - q_k| \leq C_1 \delta_k |q_k| = C_1 \mu^{\alpha k} \delta_*|q_k|.\]

Then we have the result of the theorem because \eref{induction1} and \eref{induction2} imply
\[ \|\grad u\|_{\underline{L}^2(B_{\mu^kR})} \leq C|q_j| \leq C|q|\]
and the estimate for intermediate values $\mu^{k+1}R \leq r < \mu^k R$ also holds up to an additional factor of $\mu^{-d/2}$ which is also universal.

\end{proof}
\begin{remark}\label{r.dini-mod-replace-func-lip}
The previous proof will also go through when the almost minimality hypothesis \eref{almostminimal} is replaced by
\begin{equation}\label{e.almostminimal-omeg}
 J(u,B_r) \leq (1+\omega(r/R_0))J(v,B_r) + \omega(r/R_0)|B_r| 
 \end{equation}
 (for $v$ the $a$-harmonic replacement in $B_r$) as long as $\omega^{\frac{1}{2}}$ is a Dini modulus, in the sense that $\int_0^1\omega(t)^{\frac{1}{2}}\frac{dt}{t} < +\infty$, with $\omega(1) \leq 1$.  The inductive hypotheses \eref{induction1} and \eref{induction2} with $\delta_k=\delta_* \mu^{k\alpha}$ are replaced by 
 \[\delta_k = \delta_*\omega(R/R_0)^{-\frac{1}{2}}\omega(\mu^kR/R_0)^{\frac{1}{2}}.\]  This choice is so that later in \eref{hypotheses-check-1}, when $\sigma_k = \omega(\mu^kR/R_0)$,
  we will have
   \[\sigma_k = \delta_k^2\delta_*^{-2} \omega(R/R_0) \leq c_0\delta_k^2\]
   as long as $R/R_0$ small enough depending on $\omega$ so that $\omega(R/R_0) \leq c_0 \delta_*^2$.  The purpose of the Dini modulus assumption is to guarantee that $\sum_{k=1}^\infty \delta_k < + \infty$.
\end{remark}
\subsection{Step 3: Iteration}

\begin{theorem}
If $u$ is a minimizer of $J$ in $B_R(0)$ then for all $1 \leq r \leq R$
\[ \|\grad u\|_{\underline{L}^2(B_{r}(0))} \leq C(\|\grad u\|_{\underline{L}^2(B_R(0))}+1).\]
\end{theorem}
\begin{proof}
 Let $\eta>0$ small enough and $M\geq 2C_0$ large so that the statement of \lref{DSLIPP1} holds with $\delta = C_0^{-1}\delta_*$ constants $C_0$ from \lref{DSLIPP1} and $\delta_*$ from \lref{DSinteriorLIP}.   Define
\[ b_k = \|\grad u\|_{\underline{L}^2(B_{\eta^kR})}. \]

Consider the inequality
\begin{equation}\label{e.lipiteration}
 b_k \leq \eta^{-d/2}M + 2^{-k}b_0.
 \end{equation}
This is true for $k=0$, suppose that the statement holds for $0 \leq k \leq j$.  If $\eta^{j+1}R \leq 1$ we are done so we can assume $\eta^{j+1}R \geq 1$.

If $b_j \leq M$ then
\[ b_{j+1} \leq \eta^{-d/2}b_j \leq  \eta^{-d/2}M \]
so \eref{lipiteration} holds for $k=j+1$.

Otherwise $b_j \geq M$ and so \lref{DSLIPP1} applies and either
\[ b_{j+1} \leq \frac{1}{2}b_j \leq \eta^{-d/2}M + 2^{-(j+1)}b_0,\]
so \eref{lipiteration} holds for $k=j+1$, or
\[ \|\grad u  - (q + \grad \chi_q)\|_{\underline{L}^2(B_{\eta^{j+1}R})} \leq C_0^{-1}\delta_*b_j\]
for some
\[ C_0^{-1}b_j \leq |q| \leq C_0 b_j.\]
Since $b_j \geq M \geq2C_0$ we can apply \lref{DSinteriorLIP} to find
\[ \|\grad u\|_{\underline{L}^2(B_{r})} \leq Cb_j \leq C(\eta^{-d/2}M + 2^{-j}b_0)\]
for all $1 \leq r \leq \eta^{j+1}R$.
\end{proof}

\subsection{Local regularity} Finally we mention that when the $C^{0,\gamma}$ hypothesis is added, \aref{a2}, we can also recover a true Lipschitz estimate from \tref{main1}.

\begin{corollary}
Suppose \aref{Q1}, \aref{a1}, and $a \in C^{0,\gamma}$.  If $u$ is $(R_0,\beta)$-almost minimal in $B_R(0)$ for some $R_0 \geq R > 0$ then
\[ |\grad u(0)| \leq C(1+\|\grad u\|_{\underline{L}^2(B_R)})\]
for $C\geq 1$ depending on $\Lambda$, $d$, $\beta$, $\gamma$, and $[a]_{C^{0,\gamma}}$.
\end{corollary}
\begin{proof}
If $R \geq 2 r_0$ then \tref{main1} implies that
\[ \|\grad u\|_{\underline{L}^2(B_{r_0}(x_0))} \leq C(d,\Lambda,\beta)(1+ \|\grad u\|_{\underline{L}^2(B_R)})\]
for all $x_0 \in B_{R/2}$.  If $R \leq 2r_0$ we can just do the below argument in $B_R$ instead of in $B_{r_0}$.

Without loss we can consider the case $x_0 = 0$, define the functional centered at $0$
\[ \tilde{J}_0 (v,U) = \int_U a(0) \grad v \cdot \grad v  + Q(x)^2{\bf 1}_{\{v>0\}}\ dx.\]
Note that
\[ |\tilde{J}_0 (v,B_r(0)) - J (v,B_r(0))| \leq [a]_{C^{0,\gamma}}r^\gamma \| \grad v\|_{L^2(B_r(0) \cap \{v>0\})}^2 \]
by ellipticity the gradient norm on the right hand side can be bounded above by $J(v,B_r)$ or by $\tilde{J}_0(v,B_r)$.  

So we can compute writing $B_r$ for $B_r(0)$
\begin{align*} 
\tilde{J}_0 (u,B_r) &= J(u,B_r) +( \tilde{J}_0 (u,B_r)-J(u,B_r)) \\
&\leq (1+(r/R_0)^\beta)J(v,B_r) +Cr^\gamma J(u,B_r)\\
&\leq (1+(r/R_0)^\beta)(1+Cr^\gamma)J(v,B_r)\\
&\leq (1+(r/R_0)^\beta)(1+Cr^\gamma)^2\tilde{J}_0(v,B_r)
\end{align*}
on the second inequality we used the almost minimality of $u$.  Thus $u$ has an almost minimality property for $\tilde{J}_0$ and so we can apply \cite{DeSilvaSavin}[Theorem 1.1] to get
\[ |\grad u(0)| \leq C(1+\|\grad u\|_{\underline{L}^2(B_{r_0})})\leq C(1+ \|\grad u\|_{\underline{L}^2(B_R)}).\]
\end{proof}
\section{Initial free boundary regularity}
In this section we explain the non-degeneracy at the free boundary of almost minimizers, see \sref{nondegen}.  Some inputs from homogenization theory are needed, and since we want to consider almost minimizers we again follow the arguments introduced by De Silva and Savin \cite{DeSilvaSavin}.  Non-degeneracy then implies a Hausdorff dimension $(d-1)$ estimate of the free boundary, see \sref{Hausdorffdim}.  Together with the Lipschitz estimate proved previously this is sufficient domain regularity to apply quantitative homogenization estimates in $\{u>0\}$, see \sref{recallquanthom} later. 

The following hypotheses on $u  \in H^1(U)$ non-negative will be used in this section.

\begin{enumerate}[label = $\bullet$]
\item (Lipschitz estimate)
\begin{equation}\label{e.lipschitzhyp}
 \|\grad u\|_{L^{\infty}(U)} \leq L.
 \end{equation}
\item ($\sigma$-almost minimal)  There is $\sigma\geq0$ so that for any ball $B_\rho \subset U$
\begin{equation}\label{e.sigmaminhyp}
 J(u,B_\rho) \leq J(v,B_\rho)+\sigma |B_\rho| \ \hbox{ for all } \ v \in u + H^1_0(B_\rho).
 \end{equation}
\end{enumerate}

The interior Lipschitz estimate of almost minimizers has been proved already in \sref{Lipschitz}, and will be considered a hypothesis in this section.

Note that given the Lipschitz estimate the almost minimality condition above follows from the almost minimality of the type \eref{almostminimal} via
\[J(u,B_r) \leq (1+(r/R_0)^\beta)J(v,B_r) \leq J(v,B_r) + (r/R_0)^\beta(\Lambda^2 + L^2)|B_r|  \]
for all $v \in u + H^1_0(B_r)$.

\subsection{Non-degeneracy}\label{s.nondegen} Under \eref{lipschitzhyp} and \eref{sigmaminhyp} we will show that $u$ is non-degenerate at its free boundary.  More precisely, if $\sigma$ is sufficiently small, then for any $x \in \partial \{u>0\}$ and $r>0$ so that $B_r(x) \subset U$ then
\[ \sup_{B_r(x)} u \geq c r.\]

If we were considering minimizers we could follow the original argument of Alt and Caffarelli \cite{AltCaffarelli} with some small inputs from homogenization for the oscillatory operator. Specifically we would use a Lipschitz estimate up to the boundary for the $a$-harmonic function interpolating between $1$ on $\partial B_r$ and $0$ on $\partial B_{r/2}$ in the annulus $B_r \setminus B_{r/2}$. 

To consider almost minimizers we are instead following the line of arguments by De Silva and Savin \cite{DeSilvaSavin}. Still we view the adaptations as fairly natural,  if a bit technical, using $a$-harmonic and $\overline{a}$-harmonic replacements and quantitative homogenization theory in $B_r$.  Note that, importantly, we are not using quantitative homogenization theory for the domain $\{u>0\}$ yet.

\begin{lemma}[Weak non-degeneracy]\label{l.weaknondegen}
Suppose $u>0$ and $L$-Lipschitz in $B_r$ and
\[J(u,B_r) \leq J(v,B_r)+\sigma |B_r| \ \hbox{ for all } \ v \in u + H^1_0(B_r)\]
then there is $\sigma_0$, depending on $L$ and universal parameters, and $c>0$, universal, so that if $\sigma \leq \sigma_0$ then
\[ u(0) \geq c r.\]
Furthermore if there is $x \in \partial B_r(0) \cap \partial \{u>0\}$ then
\[ \sup_{\partial B_{(1-\eta)r}} u \geq (1+\delta)u(0)\]
for some constants $\delta$ and $\eta>0$ depending on $L$ and universal parameters.
\end{lemma}

For the second part we will need the following technical Lemma.
\begin{lemma}\label{l.poissonkernel}
Given $r_0 \geq 1$ there is a constant $c$ depending on $r_0$ and universal parameters so that the Poisson kernel at $0$,  $p(y) = P_{B_r}(0,y)$, for the operator $- \grad \cdot(a(x) \grad\cdot)$ in $B_r(0)$ for any $0 < r \leq r_0$ has
\[ p(y) \geq c \frac{1}{|\partial B_r|}.\]
\end{lemma}
The proof of is postponed to the Appendix, \sref{poissonkernel}. This is the only place that we use the assumption that $a \in C^{0,1}$ as opposed to $C^{0,\gamma}$.

\begin{proof}[Proof of \lref{weaknondegen}]
Note: constants in this proof may depend on universal parameters and also on $L$.  We will try to explicitly indicate the $L$ dependence each time it appears, afterwards constants which depend on $L$ will be denoted with $k$ and $K$ and constants which are universal and don't depend on $L$ will be denoted $c$ and $C$.  The values of $k$ and $K$ may change from line to line as is the convention with $c$ and $C$.

1. Let $h$ be the $a$-harmonic replacement of $u$ in $B_r$.   Then, since $u>0$ in $B_r$ and $h > 0$ in $B_r$ (strong maximum principle),
\[ J(h,B_r) = E(h,B_r) +\int_{B_r} Q(x)^2 \ dx \leq E(u,B_r)+\int_{B_r} Q(x)^2 \ dx = J(u,B_r) \leq J(h,B_r)+\sigma |B_r| \]
so subtracting the two energies and using the $a$-harmonic replacement property we find
\[ \int_{B_r} |\grad u - \grad h|^2 dx \leq C\sigma |B_r|.\]
By Poincar\'e
\[ \int_{B_r} (u-h)^2 \ dx \leq C\sigma r^2|B_r|.\]
By the interior Lipschitz estimate \tref{hominteriorreg} for $a$-harmonic functions and maximum principle
\[ \|\grad h\|_{L^{\infty}(B_{\frac{3r}{4}})} \leq \frac{C}{r}\textup{osc}_{B_r} h \leq  \frac{C}{r}\textup{osc}_{B_r} u \leq CL.\]
So 
\begin{equation}\label{e.uminush}
 \| u - h\|_{L^\infty(B_{r/2})}^{d+2} \leq K\sigma r^2|B_r|.
 \end{equation}

Let $\phi$ be a standard smooth cutoff $\equiv1$ outside $\partial B_{r/2}$ and $\equiv0$ in $B_{r/4}$ with $|\grad \phi | \leq C/r$. Using, as earlier, that $h$ and $u$ are positive in $B_r$
\[ J(h,B_r) \leq J(u,B_r) \leq J(h\phi,B_r)+\sigma |B_r|.\]
From this point we argue differently for large and small $r>0$.  

For $r \geq r_0$ large to be specified let $\overline{h}$ be the $\overline{a}$-harmonic replacement of $h$ in $B_r$ then, by \tref{quanthomregulardomain},
\[|E(h,B_r) - E_0(\overline{h},B_r)|+\frac{1}{r^2}\|h- \overline{h}\|_{L^2(B_r)}^2 \leq CL^2r^{-\alpha}|B_r|. \]
Note the same inequality holds for $J(h,B_r) - J_0(\overline{h},B_r)$ by positivity. By similar arguments to above for \eref{uminush}, using the interior Lipschitz estimate,
\begin{equation}\label{e.hminusbarh}
 \|h-\overline{h}\|_{L^\infty(B_{r/2})}^{d+2} \leq K r^{2-\alpha}|B_r|.
 \end{equation}

Since $\overline{h} \geq 0$ solves a constant coefficient equation in $B_r$ (i.e. it is harmonic under a linear non-degenerate change of variables) we have by Harnack inequality
\[ \|\bar{h}\|_{L^{\infty}(B_{3r/4})} \leq C\bar{h}(0)\]
and so for $h$ we have
\[ \|h\|_{L^{\infty}(B_{3r/4})}\leq Ch(0) + K r^{1-\frac{\alpha}{d+2}}.\]
 and also using the interior Lipschitz estimate for $a$-harmonic functions \tref{hominteriorreg}
  \[\|\grad h\|_{L^\infty(B_{r/2})} \leq Cr^{-1}\|h\|_{L^\infty(B_{2r/4})} \leq Cr^{-1}h(0) + K r^{-\frac{\alpha}{d+2}}.\]
So, using several of the previously established inequalities,
\begin{align}
 E(h,B_r)+\int_{B_r} Q(x) \ dx &\leq  J(h,B_r) \notag\\
 &\leq J(h\phi,B_r)+\sigma |B_r| \notag \\
 &=E(h\phi,B_r)+\int_{B_r \setminus B_{r/4}} Q(x) \ dx +\sigma |B_r| \notag\\
 &\leq E(h,B_r)+\int_{B_r \setminus B_{r/4}} Q(x) \ dx \notag \\
 & \quad \quad \cdots +\left[C(\frac{1}{r^2}\|h\|_{L^{\infty}(B_{r/2})}^2 +\|\grad h\|^2_{L^\infty(B_{r/2})})+\sigma\right]|B_r|\notag\\
  &\leq E(h,B_r)+\int_{B_r \setminus B_{r/4}} Q(x) \ dx  +\left(C\frac{1}{r^2}h(0)^2+Kr^{-\frac{\alpha}{d+2}}+\sigma\right)|B_r|\notag\label{e.phicutoffh}
 \end{align}

Rearranging this inequality we find
\[ h(0)^2 \geq c r^2(\frac{1}{|B_r|}\int_{B_{r/4}} Q(x) dx - K r^{-\frac{\alpha}{d+2}}-\sigma) \geq cr^2(\Lambda^{-1}4^{-d} - Kr^{-\frac{\alpha}{d+2}}-\sigma)\]

So as long as $Kr^{-\frac{\alpha}{d+2}}+\sigma$ is sufficiently small depending on universal parameters we find
\[ h(0) \geq c r.\]
Then plugging in \eref{uminush} and \eref{hminusbarh}
\[ u(0) \geq (c- K\sigma^{\frac{1}{d+2}} - Kr^{-\frac{\alpha}{d+2}})r \geq \frac{1}{2}c r\]
This holds for $r \geq r_0$ and $\sigma \leq \sigma_0$ where $r_0$ and $\sigma_0$ depend on universal constants and on the Lipschitz constant $L$.  

Next we argue for $r \leq r_0$.  Note that we can actually argue for $r \leq r_1$ small, to be determined, and then for $r_1 \leq r \leq r_0$ we simply use that $B_{r_1} \subset B_r$ and so 
\[ u(0) \geq c r_1 \geq c \tfrac{r_1}{r_0} r\]
and the additional factor $r_1/r_0$ still depends only on $L$ and universal parameters.

 The small $r_1$ case follows a very similar line of argument to the previous using the $a(0)$-harmonic replacement $\tilde{h}$ in $B_r$ (instead of $\overline{a}$-harmonic replacement) along with the estimate
 \[ \|h - \tilde{h}\|_{\underline{L}^2(B_r)}^2 \leq Cr^2\|\grad(h-\tilde{h})\|_{\underline{L}^2(B_r)}^2\leq CL^2r^2r^{2\gamma} |B_r|\]
 which follows from the standard energy estimate
 \[ \|\grad(h-\tilde{h})\|_{L^2(B_r)} \leq C\|a - a(0)\|_{L^\infty(B_r)}\|\grad \tilde{h}\|_{L^2(B_r)}\leq Cr^{\gamma}(\int_{\partial B_r}|u||\grad u| )^{1/2} \leq Cr^\gamma L|B_r|^{1/2}\]
 using that $a$ is $C^{0,\gamma}$.

2. Since $u$ is Lipschitz in $B_r$, $\inf_{B_r} u = 0$, and $u(0) \geq cr$ (from the first part of the proof) we can also conclude that $u(x) \leq \frac{1}{4}u(0)$ on a constant $k(L)$ fraction of $\partial B_{(1-\eta)r}$ for some $\eta(L) > 0$.  Since
 \[ \sup_{B_{(1-\eta)r}} |u - \bar{h}| \leq K(\sigma^{\frac{1}{d+2}} + r^{-\frac{\alpha}{d+2}})r \leq K(\sigma^{\frac{1}{d+2}} + r^{-\frac{\alpha}{d+2}})u(0)\]
 when $\sigma \leq \sigma _0(L)$ and $r \geq r_0(L)$ also $\bar{h}(x) \leq \frac{1}{2}\bar{h}(0)$ on a constant fraction of $\partial B_{(1-\eta)r}$.  Then by the  Poisson kernel formula and the argument in Caffarelli \cite{CaffarelliIII}[Lemma 7] we must have 
 \[ \sup_{x \in B_{(1-\eta)r}} \bar{h}(x) \geq (1+2\delta) \bar{h}(0)\]
 and choosing $\sigma_0$, and $r_0$ smaller if necessary we get the same for $u$ via
\begin{align*}
 \sup_{x \in B_{(1-\eta)r}} u(x) &\geq \sup_{x \in B_{(1-\eta)r}} \bar{h}(x) - K(\sigma^{\frac{1}{d+2}} + r^{-\frac{\alpha}{d+2}})u(0) \\
 &\geq (1+2\delta) \bar{h}(0)-K(\sigma^{\frac{1}{d+2}} + r^{-\frac{\alpha}{d+2}})u(0) \\
  &\geq (1+2\delta) u(0)-K(\sigma^{\frac{1}{d+2}} + r^{-\frac{\alpha}{d+2}})u(0) \\
  & \geq (1+\delta) u(0).
 \end{align*}
 For $r \leq r_0$ we just do $a$-harmonic replacement $h$ with the estimate
 \[ \sup_{B_{(1-\eta)r}} |u - h| \leq K\sigma^{\frac{1}{d+2}}r \leq K\sigma^{\frac{1}{d+2}}u(0)\]
 and then apply the Poisson kernel lower bound \lref{poissonkernel}, depending on $r_0(L)$, and the same argument from \cite{CaffarelliIII}[Lemma 7] 
 \[ \sup_{x \in B_{(1-\eta)r}} h(x) \geq (1+2\delta) h(0)\]
 where $\eta$ is the same as above, and $\delta$ will now depend on the Poisson kernel lower bound from \lref{poissonkernel}.  Using the supremum $|u-h|$ estimate as before and choosing $\sigma_0$ smaller again if necessary we conclude.
\end{proof}

\begin{lemma}[Strong non-degeneracy]\label{l.nondegen}
Suppose $u$ is $L$-Lipschitz \eref{lipschitzhyp}, and satisfies \eref{sigmaminhyp} in $B_r(0)$.  There are $\sigma_0,c>0$ depending on $L$ and universal parameters so that if $ 0 < \sigma \leq \sigma_0$, and if $0 \in \partial \{u>0\} \cap U$ then
\[ \sup_{B_r(0) \cap U} u \geq c r.\]
\end{lemma}
\begin{proof}
As in \cite{DeSilvaSavin}[Lemma 3.5] the proof follows a standard argument applying \lref{weaknondegen}, see originally \cite{CaffarelliIII}[Lemma 7]. We just need the following claim constructing a polygonal chain along which $u$ grows linearly.

Claim: There are $\delta>0$ small and $C \geq 1$ depending on $L$ and universal constants so that: given $x_0 \in \{u>0\} \cap B_r$ (near the origin) there is a sequence $x_k \in B_r \cap \{u>0\}$ with
\[ u(x_{k+1}) \geq (1+\delta)u(x_k)\]
and
\[ |x_{k+1}-x_k| =  cd(x_k,\partial\{u>0\}).\]

This is exactly what we proved in the second part of \lref{weaknondegen}.
\end{proof}
\begin{remark}\label{r.harnack-chain-comment}
Note that Caffarelli's Harnack chain argument requires starting from a possibly arbitrarily small radius ball so we do need \lref{weaknondegen} for all values of $r >0$ and so we do need the $a \in C^{0,\gamma}$ assumption here.  Maybe that assumption could be removed here by using some stronger notion of ``bulk" free boundary point.
\end{remark}

\subsection{Hausdorff dimension of the free boundary}\label{s.Hausdorffdim}

In this section we recall estimates of the Hausdorff dimension / measure of $\partial \{u>0\}$.  
\begin{lemma}[Free boundary strip energy bound]\label{l.bdrystripenergy}
Suppose that $u \in H^1(B_{2r})$ with $r \geq 1$ has
\[ J(u,B_{2r}) \leq J(v,B_{2r})+\sigma|B_{2r}| \ \hbox{ for all } \ v \in u + H^1_0(B_{2r}) \]
Then for $1 \leq t \leq r$
\[ J(u,\{0 < u \leq t\} \cap B_{r}) \leq C\left[\sigma+\frac{t}{r}(1+\|\grad u\|_{\underline{L}^2(B_{2r})})\right]|B_r|\]
with $C \geq 1$ universal.
\end{lemma}

\begin{proof}
The proof is standard, see \cite{VelichkovNotes}[Lemma 5.6], and only uses the ellipticity of the coefficient fields $a(x)$ and $Q(x)$. The idea is to use the following energy competitor: take $0 \leq \phi \leq 1$ be a smooth cutoff function $\phi = 0$ in $B_{r}$, $\phi = 1$ in a neighborhood of $\partial B_{2r}$ and $|\grad \phi| \leq C/r$ and define
\[ v(x) = (1-\phi(x)) (u(x) - t)_+ + \phi(x) u(x).\]
Using the almost minimality property and some computations give the result.
\end{proof}

\begin{lemma}[Free boundary Hausdorff dimension]\label{l.boxcovering}
Suppose that $u$ Lipschitz with constant $L \geq 1$ and $\ell$-non-degenerate in $B_{2r}$.   Then there is a universal $C \geq 1$ so that for any covering $(\Box_i)_{i \in I}$ of $\partial \{u>0\} \cap B_{r}$ by almost disjoint boxes of side length $1 \leq t \leq c(d)r$
\[ \#I \leq C(d,\Lambda)\left(\tfrac{L}{\ell}\right)^d t^{-d}J(u,\{ 0 < u \leq c(d)tL\} \cap B_{2r})\]
\end{lemma}
Of course \lref{bdrystripenergy} can then be used to bound the right hand side in the inequality obtained here.
\begin{proof}
Let $\Box$ be a cube with side length $t$ such that $\partial \{u>0\} \cap \Box \neq \emptyset$ and $2 \Box$ be the cube with the same center dilated by a factor of $2$.  Then by the non-degeneracy hypothesis
\[ \sup_{x \in 2 \Box}u(x) \geq c(d)\ell t  \]
By the Lipschitz assumption there is a ball of radius $c(d)\frac{\ell}{L}t$ inside of $3 \Box$ with
\[ u(x) \geq c(d)\ell t \ \hbox{ in } \ B_{c(d)\frac{\ell}{L}t} \subset 3 \Box. \]
For each $\Box_i$ from the collection in the statement let $B_{\frac{\ell}{2L}}(y_i) \subset 3 \Box_i$ be the corresponding cube described above.  Since $\Box_i$ are almost disjoint there is a dimensional constant $C(d)$ so that each $x \in B_r$ is in at most $C(d)$ of the $(B_{\frac{\ell}{2L}}(y_i))_{i \in I}$. Thus
\[ \#I =  \sum_{i \in I} 1 = \sum_{i \in I} t^{-d}|\Box_i| \leq C(d)(\tfrac{L}{\ell t})^d\sum_{i \in J} |B_{c(d)\frac{\ell}{L}t}(y_i)| \leq C(d)(\tfrac{L}{\ell t})^d | \cup_{i \in I}B_{c(d)\frac{\ell}{L}t}(y_i)|\]
and
\[  | \cup_{i \in I}B_{c(d)\frac{\ell}{L}t}(y_i)| \leq \Lambda J(u,\cup_i 3 \Box_i) \leq \Lambda J(u,\{0 < u < c(d)Lt\} \cap B_{r+3\sqrt{d}t})\]
since $t \leq c(d) r$ we have $B_{r+3\sqrt{d}t} \subset B_{2r}$.
\end{proof}

\section{H\"older rate of homogenization}\label{s.suboptimal-hom}

In this section we combine the previous regularity results to show a sub-optimal quantitative homogenization rate for almost minimizers.  We will apply quantitative homogenization results for Dirichlet boundary value problems.

\subsection{Quantitative homogenization of Dirichlet problems}\label{s.recallquanthom}

We recall some results about quantitative homogenization results for Dirichlet boundary value problems in bounded domains $U \subset \R^n$. Recall the energy functionals
\[ E(u,U) = \int_U a(x) \grad u \cdot \grad u \ dx \ \ \hbox{ and } \ E_0(u,U) = \int_U \bar{a} \grad u \cdot \grad u \ dx \]
originally introduced in \sref{homrecall}.

Given a bounded domain $U$ in $\R^d$ we consider a function $u \in H^1(U)$ satisfying
\[ \|\grad u\|_{L^\infty(U)} \leq L\]
in our application $u$ will be an almost minimizer of the functional $J$ and $U$ will be $\{u>0\}$, but that information is not needed for the present statements.

We also consider a function $u_0 \in H^1(U)$ satisfying
\[ \|\grad u_0\|_{L^\infty(U)} \leq L\]
and
\[ E_0(u_0,U) \leq E(v,U) \ \hbox{ for all } \ v \in u_0 + H^1_0(U).\]
The function $u_0$ will arise by doing a $J_0$ minimizer replacement of an almost minimizer and $U$ will be $\{u_0>0\}$, but that information is not needed for the present statements.

We will take advantage of the a-priori Lipschitz estimate.  It is convenient to work directly with the explicit upscalings/downscalings, which usually appear as intermediate tools in quantitative homogenization proofs, instead of with the usual comparison with the $\bar{a}$-harmonic replacement.  See \rref{fixed-domain-hom-comment} below.

{$\circ$ \bf Specific choice of upscaling.}  Let $\varphi : \R \to [0,1]$ be the piecewise linear continuous cutoff function with $\varphi(t)\equiv 0$ for $t \leq 1$, $\varphi \equiv 1$ for $t \geq 2$, and $|\varphi'| \leq 2$.  Let $t \ll r$ be a scale to be determined define
\[ \xi_t(x) = \dashint_{x + B_t} u(y) \ dy\]
and
\begin{equation}
 \bar{u}(x) = \varphi(\tfrac{d(x)}{t}) \xi_t (x) + (1-\varphi(\tfrac{d(x)}{t}))u(x)
 \end{equation}
where $d(x)$ is the distance function to the complement of $U$.  It is easy to check
\[ \|\grad \xi_t \|_{L^\infty(U)} \leq \|\grad u \|_{L^\infty(U)}\]
and
\[ \|u - \bar{u}\|_{L^\infty(U)} \leq \|\grad u \|_{L^\infty(U)}t.\]
Since
\[ \grad \bar{u} = \varphi(\tfrac{d(x)}{t}) \grad \xi_t (x) + (1-\varphi(\tfrac{d(x)}{t}))\grad u(x) + \frac{1}{t}\grad d(x)\varphi'(\tfrac{d(x)}{t})( \xi_t (x) - u(x))\]
and $|\grad d| \leq 1$
\[\|\grad \bar{u}\|_{L^\infty(U)} \leq 2\|\grad u \|_{L^\infty(U)}. \]

{$\circ$ \bf Specific choice of downscaling.}  Let $\varphi : \R \to [0,1]$ be a fixed smooth cutoff function with $\varphi(t)\equiv 0$ for $t \leq 1$, $\varphi \equiv 1$ for $t \geq 2$, and $|\varphi'| \leq C$.  Let $t \ll r$ be a scale to be determined and
\begin{equation}
 \tilde{v}(x) = \varphi(\tfrac{d(x)}{t}) [v(x) +  \chi_{\grad v(x)}(x) \grad v(x)] + (1-\varphi(\tfrac{d(x)}{t}))v(x)
 \end{equation}
where $d(x)$ is the distance function to the complement of $U$ and $\chi_q$ are the correctors defined in \sref{homrecall}.

\begin{proposition}[Quantitative homogenization]\label{p.quanthom}
Suppose that $U$ is a bounded domain and let $u$, $\bar{u}$, $u_0$, and $\tilde{u}_0$ as above. Let $t \geq 1$ be an integer, and let $U_t$ be the union of the $t \Z^d$ lattice cubes which are contained in $U$.   There is $C \geq 1$ depending on $d$ and $\Lambda$ so that
\[ E_0(\bar{u},U) \leq E(u,U)+ CL^2 \left[\frac{1}{t^{\frac{1}{2}}}|U|+|U \setminus U_t|\right] \]
and
\[ \left| E(\tilde{u}_0 ,U) - E_0(u_0,U)\right| \leq  CL^2\left[\frac{1}{t}|U|+|U \setminus U_t|\right].\]
\end{proposition}

\begin{remark}\label{r.fixed-domain-hom-comment}
We are presenting a somewhat atypical scenario for quantitative Dirichlet data homogenization because we have assumed an a-priori Lipschitz estimate up to the domain boundary.  This usually would require significant domain regularity to prove, but we will know it from \tref{main1} for $J$ almost minimizers.  

On the other hand we also know \emph{less} domain regularity than one typically assumes for such quantitative homogenization results.  A typical assumption is that the domain $U$ is Lipschitz, see for example \cite{AKM}[Theorem 1.12], but we are aiming to prove this.  Actually an inner and outer density bound is sufficient when the equation does not have a right hand side, see \cite{AKM}.  We can indeed prove inner and outer density bounds at this stage, but the result is not strictly needed for our main theorems so we have postponed it to the appendix in \sref{density}. 

Results in the literature also typically estimate the difference with the $\bar{a}$-harmonic replacement. With the amount of domain regularity we currently know, the $\bar{a}$-harmonic replacement is actually less regular up to the domain boundary.    It will not necessarily be Lipschitz up to the boundary of $U$ (insufficient domain regularity) and we won't have an $L^\infty$ convergence estimate (insufficient domain regularity again).
\end{remark}

The proof follows standard lines.  In particular, the statements and arguments in \cite{Shen}[Ch. 3.2 and 3.3] are very similar to what we will do here. We will refer there in some places but we didn't find any particular statement there to cover the entirety of \pref{quanthom}.  We also emphasize that we are not making any attempt to find optimal rate here, it is not needed for our purposes.

\begin{proof}[Proof of \pref{quanthom}]
We will make use of the dual energy introduced by \cite{ArmstrongSmart}
\begin{equation}\label{e.mumin}
 \mu(U,q) = \inf \{ \dashint_{U} \tfrac{1}{2}\grad v \cdot  a(x) \grad v - q \cdot \grad v \ dx : \ v \in H^1(U)\}.
 \end{equation}
This quantity has the following large scale limit
\[ |\mu(B_t,q)+\frac{1}{2}q \cdot \bar{a}^{-1}q | \leq C(\Lambda,d)|q|^2 \frac{1}{t^{1/2}}\]
which we now justify.  We only are interested in the lower bound, the upper bound is immediate from testing the appropriate corrected linear function (below).

For $\tau \in [0,1)^d$ let $v$ be the minimizer for \eref{mumin} solving
\[ \begin{cases}
-\grad (a(x) \grad v) = 0 & \hbox{ in } \ B_t \\
 n \cdot (a(x) \grad v_\tau-q) = 0 & \hbox{ on } \ \partial B_t
 \end{cases}\]
with mean zero on $B_t$ and let
\[ v_0 = p \cdot x+\chi_p(x) \ \hbox{ with } \ \bar{a}p = q\] 
note that $p \cdot x$ solves the effective version for the previous Neumann problem.  Note 
\[ |\dashint_{B_t} \tfrac{1}{2}\grad v_0 \cdot a(x)\grad v_0 -q \cdot \grad v_0\ dx -( \frac{1}{2} p\cdot \bar{a}p - q \cdot p)| \leq Ct^{-1}|q|^2.\]
Then call
\[ w = v(x) - p\cdot x - \chi_p(x) \phi(x)\]
where $\phi$ is a standard cutoff which goes from $1$ to $0$ in a unit neighborhood of $\partial B_t$. By \cite{Shen}[Theorem 3.3.4]
\[ \|\grad w\|_{L^2(B_t)} \leq Ct^{-1/2}|q|\]
and so
\[ \dashint_{B_t} \tfrac{1}{2}\grad v \cdot  a(x) \grad v - q \cdot \grad v \ dx \geq \dashint_{B_t} \tfrac{1}{2}\grad v_0 \cdot  a(x) \grad v_0 - q \cdot \grad v_0 \ dx -Ct^{-1/2}|q|. \]

 Part 1 (estimate of the upscaling). Define, for $x \in U_t$
 \[\xi(x) = \dashint_{x+B_{t}} u(y) \ dy\]
 and
 \[ \bar{u}(x) = \varphi_t(x) \xi(x) + (1-\varphi_t(x)) u(x).\]
 Call
 \[ p(x) = \grad \xi(x) = \int_{x+B_{t}} \grad u(y) \ dy\]
 and
 \[ q(x) = \bar{a}p(x)\]
 Note that $|q(y)| \leq \Lambda L $.  Now the dual quantity $\mu$ is exactly useful for lower bounds of the energy
 \begin{align*} \int_{U}  \tfrac{1}{2}\grad u \cdot a(x)\grad u \ dx  &\geq \int_{U_t} \dashint_{y+B_{t}}  \tfrac{1}{2}\grad u(x) \cdot a(x)\grad u(x) dx  dy \\
 &\geq \int_{U_t} \mu(B_{t},q(y))  +q(y) \cdot p(y)dy \\
 &=\int_{U_t} \mu(B_t,q(y))  +  p(y) \cdot \bar{a} p(y)dy\\
 & \geq \int_{U_t} \tfrac{1}{2}p(y) \cdot \bar{a} p(y)dy + \int_{U_t} \left[\frac{1}{2}q(y) \cdot \bar{a}^{-1}q(y) + \mu(B_{t},q(y))\right] \ dy \\
 & = E_0(\bar{u},U_t) - CL^2t^{-1/2}|U|
 \end{align*}
 Then we need to estimate the difference between the homogenized energy of $\bar{u}$ on $U_t$ and on $U$
 \begin{align*} \int_{U_t}  \tfrac{1}{2}p(x) \cdot \bar{a} p(x) dx -   \int_{U}  \tfrac{1}{2}\grad\bar{u} \cdot \bar{a} \grad\bar{u}dx&= -\int_{U \setminus U_t} \grad \bar{u}\cdot \bar{a} \grad \bar{u} dx \\
 &\geq -\Lambda L^2 |U \setminus U_t|
 \end{align*}
 Together we have
 \[ E_0(\bar{u},U) \leq E(u,U) + CL^2\left[ t^{-1}|U| +|U \setminus U_t|\right]. \]
 
 Part 2 (estimate of the downscaling). This is basically the same as \cite{Shen}[Theorem 3.3.2].  Recall that $u_0$ is a minimizer of $E_0$ on $U$, so $\grad u_0$ is $\bar{a}$-harmonic and we have the $C^{1,1}$ estimate
 \[ \|D^2 u_0\|_{L^\infty(U_t)} \leq Ct^{-1}L.\]
 The downscaling was defined, we recall,
 \[ \tilde{u}_0(x) = u_0(x) + \chi(x) \cdot \grad u_0(x) \varphi_t(x).\]
 The gradient is, with the primary term listed first and error terms listed after,
 \[ \grad \tilde{u}_0 = [\grad u_0 +\grad \chi \grad u_0]+(1-\varphi_t)[\grad \chi \grad u_0]+\chi D^2u_0 \varphi_t + (\chi \cdot \grad u_0)\grad \varphi_t. \]
  We make note of the following bounds following from the $L$-Lipschitz hypothesis of $u_0$ and corrector bounds \eref{correctorest3}
 \[ |\grad u_0 +\grad \chi \grad u_0| \leq CL, \quad |(1-\varphi_t)[\grad \chi \grad u_0]| \leq L(1-\varphi_t), \quad |\chi D^2u_0 \varphi_t| \leq CLt^{-1}\]
 and
 \[ (\chi \cdot \grad u_0)\grad \varphi_t \leq CLt^{-1}.\]
 Let $p(x)$ be the piecewise constant function which is the average of $\grad u_0$ over the $t\Z^d$ translations of $[0,t)^d$ which contains $x$. Then
 \[ |p- \grad u_0|\varphi_t \leq CLt^{-1}\]
 by the previous $C^{1,1}$ estimate.
 
 Thus, by standard algebra with the quadratic energy,
\[ \left|E(\tilde{u}_0,U) - \int_{U_t}(p(x) + \grad \chi(x) p(x)) \cdot a(x)(p(x) + \grad \chi(x) p(x)) \ dx\right| \leq CL^2[t^{-1}|U|+|U \setminus U_t|].\]
Since $t \geq 1$ is an integer and by the definition of $\bar{a}$, \eref{baradef},
\[\int_{U_t}(p(x) + \grad \chi(x) p(x)) \cdot a(x)(p(x) + \grad \chi(x) p(x)) \ dx = \int_{U_t} p(x) \cdot \bar{a}p(x) \ dx\]
and
\[ |\int_{U_t} p(x) \cdot \bar{a}p(x) \ dx - \int_{U_t} \grad v(x) \cdot \bar{a}\grad u_0(x) \ dx| \leq CL^2t^{-1}|U|\]
so finally, using one more time the Lipschitz estimate for $E_0(u_0,U \setminus U_t) \leq CL^2|U \setminus U_t|$,
\[ |E(\tilde{u}_0,U) - E_0(u_0,U)|\leq CL^2[t^{-1}|U|+|U \setminus U_t|].\]

\end{proof}

\subsection{Rate of homogenization for $J$}
In this section we combine all the previous regularity results to obtain a rate of homogenization in the energy.  We will consider $u \in H^1(U)$ non-negative satisfying \eref{lipschitzhyp} and \eref{sigmaminhyp}.

Note that if $u$ is $L$-Lipschitz and satisfies \eref{sigmaminhyp} for $\sigma>0$ sufficiently small depending on $L$ and universal constants then $u$ satisfies the hypotheses of \lref{nondegen}, \lref{bdrystripenergy}, and \lref{boxcovering}.

\begin{lemma}[Energy convergence]\label{l.energyconv}
Let $r \geq 1$.
\begin{enumerate}[label = (\alph*)]
\item  Suppose that $u \in H^1(B_{2r})$ non-negative satisfies \eref{lipschitzhyp} and \eref{sigmaminhyp} with $0 < \sigma \leq \bar{\sigma}(L)$ sufficiently small so that \lref{nondegen} holds.  There is a regularization $\bar{u}$ of $u$, $\overline{u} \in u + H^1_0(\{u>0\} \cap B_r)$ with
\[ \|\grad \bar{u}\|_{L^\infty(B_r)} \leq L, \ \|u - \bar{u}\|_{L^\infty(B_r)} \leq Lr^{2/3}\]
and
\[  J_0(\bar{u},B_{r}) \leq J(u,B_{r})+C(d,\Lambda,L)[\sigma + r^{-1/3}]|B_r|.\]
\item Suppose that $v$ is a $J_0$-minimizer in $B_{2r}$ satisfying \eref{lipschitzhyp} in $B_{2r}$.  There is $\tilde{v} \in v + H^1_0(\{v>0\} \cap B_r)$ such that
\[ J(\tilde{v},B_{r}) \leq J_0(v,B_{r})+ C(d,\Lambda)L^{\frac{d+5}{2}}r^{-1/2}|B_r|.\]
\item Let $u$ and $\bar{u}$ as in part $(a)$. For any $w \in u + H^1_0(\{u>0\} \cap B_r)$
\[ J_0(\bar{u},B_{r}) \leq J_0(w,B_r) + C(L)[r^{-\omega}+\sigma]|B_r|\]
where $\omega$ is a dimensional constant and can be taken to be $\omega = \frac{1}{4d(d+5)+6}$ .
\end{enumerate}
\end{lemma}
\begin{remark}
This is the only place in the paper where the periodicity of $Q$ is used.  Other assumptions on the convergence of the spatial averages of $Q$ could be easily slotted in here as long as they result in an error rate which is summable over geometric sequences of scales (that property is used in the following section).
\end{remark}
\begin{remark}
The constant $\omega = \frac{1}{4d(d+5)+6}$ is certainly suboptimal, we compute it just to be careful that we indeed have a H\"{o}lder rate.  The dependence on $L$ in part (b) is followed for the sake of computing $\omega$, so, again, it is not that important.
\end{remark}
For the proof of \lref{energyconv} we will need one more technical Lemma about $J_0$ minimizers in a domain with Lipschitz Dirichlet data. The claim is that there is not too much energy concentrated near the domain boundary.
\begin{lemma}[Domain boundary strip energy bound]\label{l.domainbdryenergy}
Suppose that $u \in g + H^1_0(B_r)$, $g$ is $L$-Lipschitz, and $u$ minimizes $J_0$ over $g + H^1_0(B_r)$ then for every $\alpha < 1/2$
\[ J_0(u,B_r \setminus B_{(1-\eta)r}) \leq C(\alpha,L)\eta^{\alpha}(1 + \frac{1}{r}) |B_r|.\]
\end{lemma}

The same result is probably true for all $\alpha < 1$ but we did not need this and it would require some more work to prove.  The result would also be true for a regular domain $U$ instead of a ball, but again we did not need that.  The proof is technical and is postponed to \sref{domainbdryenergy}. There is a barrier argument to get H\"older continuity up to $\partial B_r$ and then an energy comparison argument.

\begin{proof}[Proof of \lref{energyconv}]

 In the proof throughout we will assume $L \geq 1$, constants which depend on $L$ and other universal constants will be denoted $C(L)$ or $c(L)$, universal constants which do not depend on $L$ will be denoted $C$, $c$ etc.  For parts (b) and (c) we need to be a bit careful keeping track of $L$ dependence.
 
 Part (a).  Let $\bar{\sigma}>0$ depending on universal constants and on $L$ so that \lref{nondegen} holds.  Let $U = \{u>0\} \cap B_r$ and let $\bar{u}$ be the upscaling of $u$ defined in \sref{recallquanthom} with parameter $t\geq 1$ to be chosen shortly. Let $U_t$ be the union of the $t\Z^d$ lattice cubes contained in $U$.  By the perimeter estimate \lref{boxcovering} and \lref{bdrystripenergy} 
\[ |U \setminus U_t| \leq C(L)(\sigma + \frac{t}{r})|B_r|\]
Thus to infimize the error term from \pref{quanthom} we should choose $t = r^{2/3}$
\[  \inf_{t>0}[\frac{1}{t^{\frac{1}{2}}}|U| + |U \setminus U_t|] \leq C(L)[\sigma + r^{-1/3}]|B_r|\]

Call $\mathcal{Q}$ to be the cubes of the $ \Z^d$ lattice 
\[ \mathcal{Q} = \{ k + [0,1)^d: k \in \Z^d  \}\]
and then define
\[ \mathcal{A} = \{\Box \in \mathcal{Q} : \ \Box \subset (\{u>0\}\cap B_r)\} \]
and
\[ \mathcal{B} =\{ \Box \in \mathcal{Q} : \ \Box\cap \partial (\{u>0\}\cap B_r) \neq \emptyset\}.\]
Since $u$ satisfies the conclusion of \lref{boxcovering}, 
\[ \#\mathcal{B} \leq Cr^{d-1} \leq C(L)[\sigma+\frac{1}{r}]|B_r|.\]
Note that for any $\Box \in \mathcal{A}$
\[ \int_{(\{u>0\} \cap B_r) \cap \Box} Q(x)^2 \ dx = \int_{\Box} Q(x)^2 \ dx = \langle Q^2 \rangle  =\int_{(\{\overline{u}>0\} \cap B_r) \cap \Box} \langle Q^2 \rangle \ dx  \]
 using, respectively, that $\Box \subset \{u>0\} \cap B_r$, $\Box$ is a period cell for $Q$, and that $u$ and $\overline{u}$ have the same positivity set.

Then applying \pref{quanthom} 
\begin{align*}
 J_0(\overline{u},B_r) - J(u,B_r) &= E_0(\overline{u},\{u>0\} \cap B_r)-E(u,\{u>0\} \cap B_r) \\
 & \quad \quad \cdots + \sum_{\Box \in \mathcal{A} \cup \mathcal{B}} \int_{\{u>0\} \cap \Box \cap B_1} \langle Q^2 \rangle-Q(x)^2  \ dx \\
 &\leq C(L)[\sigma + r^{-1/3}]|B_r| + \Lambda |\mathcal{B}| \\
 & \leq C(L)(\sigma+r^{-1/3})|B_r|.
 \end{align*}
 
 Part (b).  The proof follows an analogous line of arguments to part (a). Since we need to keep track of the dependence on the Lipschitz constant for the purpose of part (c).  Note that the non-degeneracy constant does \emph{not} depend on the Lipschitz constant in the case of $J_0$ minimizers, it is universal.  Let $V = \{v>0\} \cap B_r$ and let $\tilde{v}$ be the downscaling of $v$ in $V$ defined in \sref{recallquanthom}.  By the perimeter estimates \lref{boxcovering} and \lref{bdrystripenergy} (with $\sigma = 0$), now keeping more carefully the dependence on $L$,
 \[ |U \setminus U_t| \leq C L^{d+2}t^2r^{-1}|B_r| \]
 Thus to infimize the error term from \pref{quanthom} we should choose $t = L^{-(d+1)/2}r^{1/2}$
\[ \inf_{t>0}[\frac{1}{t}|U| + |U \setminus U_t|]  \leq C L^{\frac{d+1}{2}}r^{-1/2}|B_r|.\]
The remainder of the argument is similar and we end up with the estimate
\[ J(\tilde{v},B_r) \leq J_0(v,B_r) + CL^{\frac{d+5}{2}}r^{-1/2}|B_r|\]
with the additional factor of $L^2$ coming from \pref{quanthom}.

 Part (c).  Let $\overline{u}$ be the upscaling of $u$ in $B_{r} \cap \{u>0\}$ defined earlier. Let $v$ be a minimizer of $J_0$ in the class $u + H^1_0(\{u>0\} \cap B_{r})$.  Then, since $v$ is a minimizer for $J_0$ in $B_r$,
\[ J_0(v,B_{r}) \leq J_0(\overline{u},B_{r}).\]
In particular
\[ \|\grad v\|_{\underline{L}^2(B_r)} \leq C(\|\grad u\|_{\underline{L}^2(B_r)}+1) \leq CL.\]

Let $1>\eta>0$ to be chosen small and $\mathcal{Q}$ be the class of $\eta r$-lattice cubes which are contained in $B_{(1-\eta)r}$. Note that for every $Q \in \mathcal{Q}$ we have $2Q \subset B_r$. In particular by the interior Lipschitz estimate of $J_0$ minimizers $v$ is Lipschitz in $B_{(1-\eta)r}$ with
\begin{equation}
\label{e.vlipin} |\grad v(x)| \leq C(1+\|\grad v\|_{\underline{L}^2(B_{\eta r})}) \leq C\eta^{-d/2}L \ \hbox{ for } \ x \in B_{(1-\eta) r}.
\end{equation}
Let $\tilde{v}_Q$ be the downscaling of $v$ in $Q \cap \{v>0\}$ defined in \sref{recallquanthom} with scale $t = L^{-(d+1)/2}\eta^{\frac{d(d+1)}{2}}(\eta r)^{1/2}$, the scale is chosen as in (b) taking into account the Lipschitz constant of $v$ from \eref{vlipin}.  Note that at this point it is really the $\eta$ dependence and not the $L$ dependence that we want to track.  We have the estimate from (b)
\begin{equation}
J(\tilde{v}_{Q},Q)  \leq J_0(v,Q)+C(L)\eta^{-\frac{(d+5)d}{2}} (\eta r)^{-1/2}|Q|
\end{equation}

We will make an energy competitor
\[ w = v{\bf1}_{B_r \setminus \cup_{Q \in \mathcal{Q}} Q}+\sum_{Q \in \mathcal{Q}}\tilde{v}_Q{\bf 1}_{Q}\]
 which does have $w \in u + H^1_0(B_r)$ due to the matching of appropriate traces on each cube boundary. Since $u$ is a minimizer for $J$ in $B_r$, 
\[J(u,B_{r}) \leq J(w,B_r) = \sum_{Q \in \mathcal{Q} }J(\tilde{v}_{Q},Q) + J(v,B_{r} \setminus \cup_{Q \in \mathcal{Q}} Q) .\]
The boundary layer term we will bound by applying \lref{domainbdryenergy}, using that $\|\grad u \|_{L^{\infty}(B_r)} \leq L$,
\[J(v,B_{r} \setminus \cup_{Q \in \mathcal{Q}} Q) \leq J(v,B_{r} \setminus B_{(1-C(d)\eta)r}) \leq C(L)\eta^{1/4}|B_r|.\]

We also have, from part (a) and the above,
\begin{align*}
 J_0(\overline{u},B_r) &\leq J(u,B_r) + C(L)[r^{-1/2}+\sigma]|B_r| \\
 &\leq \sum_{Q \in \mathcal{Q} }J(\tilde{v}_{Q},Q)+C(L)[r^{-1/3}+\eta^{1/4}+\sigma]|B_r|
 \end{align*}
and by part (b)
\begin{align*}
\sum_{Q \in \mathcal{Q} }J(\tilde{v}_{Q},Q) 
&\leq \sum_{Q \in \mathcal{Q}}\left\{J_0(v,Q)+ C(L) \eta^{-\frac{d(d+5)+1}{2}} r^{-1/2}|Q|\right\}\\
& \leq J_0(v, B_r)+ C(L)  \eta^{-\frac{d(d+5)+1}{2}}r^{-1/2}|B_r|.
\end{align*}
Combining all the above gives
\[ J_0(\overline{u},B_{r})\leq J_0(v, B_r)+ C(L) (r^{-1/3}+\eta^{-\frac{d(d+5)+1}{2}}r^{-1/2}+\eta^{1/4}+\sigma)|B_{r}|\]
then choose $\eta = r^{-\frac{1}{2}\frac{4}{2d(d+5)+3}}$ to match the middle two error terms on the right.

\end{proof}

\section{Flat free boundaries are (large scale) Lipschitz}\label{s.flatimplieslipschitz}

In this section we will show that the free boundaries of flat minimizers of $J$ are $C^{1,\alpha}$ at intermediate scales and Lipschitz down to the unit scale.  This will prove \tref{main2} as a consequence of \lref{eplip} and \cref{lipbdry} below.  We will also show, in \sref{liouville}, a Liouville result for large scale flat minimizers of $J$ on $\R^d$ \tref{liouville}.

We will combine the homogenization result \lref{energyconv} with the following $C^{1,\alpha}$ flatness improvement result of De Silva and Savin \cite{DeSilvaSavin} for almost minimizers of $J_0$:

\begin{lemma}[De Silva and Savin \cite{DeSilvaSavin} Lemma~4.5]\label{l.dsalmostmin-iteration}
Let $u$ satisfy
 \begin{equation}\label{e.lipandenergy}
  \|\grad u\|_{L^\infty(B_1)} \leq L \ \hbox{ and } \ J_0(u,B_1) \leq J_0(v,B_1)+\sigma \ \hbox{ for all } \ v \in u + H^1_0(B_1).
  \end{equation}
For any $0 < \alpha < 1$ there exist constants $\bar{\delta}, \eta>0$ and $C \geq 1$ depending on $L$ and $\alpha$ so that if $0 \in \partial \{u>0\}$ and, for some $\nu \in S^{d-1}$,
 \[ |u(x) -\tfrac{1}{(\nu \cdot \overline{a} \nu)^{1/2}} (x \cdot \nu)_+| \leq \delta \ \hbox{ in } \ B_1 \]
 for some $0< \delta \leq \bar{\delta}$ then there is $\nu' \in S^{d-1}$ with $|\nu' - \nu| \leq C(\delta+\sigma^{\frac{1}{d+4}})$
 \[ |u(x) -\tfrac{1}{(\nu' \cdot \overline{a} \nu')^{1/2}}(x\cdot \nu')_+ | \leq \eta^{1+\alpha}(\delta+C\sigma^{\frac{1}{d+4}}) \ \hbox{ in } \ B_{\eta}(0).\]
\end{lemma}

\begin{remark}
Note that \cite{DeSilvaSavin} works with the standard Dirichlet energy, but we can transform to this case by looking at $v(x) = u(\bar{a}^{1/2}x)$.
\end{remark}

We apply \lref{dsalmostmin-iteration} to minimizers of $J$ by doing an upscaling and applying \lref{energyconv} to see that the upscaling has (almost) minimal energy for $J_0$.

First we give the one step flatness improvement.
\begin{lemma}[Improvement of flatness]\label{l.eplip1step}
 Let $u$ satisfy \eref{lipschitzhyp} and \eref{sigmaminhyp} in $B_{2r}$ with $0<\sigma \leq \bar{\sigma}$ sufficiently small so that \lref{nondegen} holds.  For any $0 < \alpha < 1$ there exist constants $\bar{\delta}, \eta>0$ depending on $L$ and $\alpha$ so that if $0 \in \partial \{u>0\}$ and, for some $\nu \in S^{d-1}$,
 \[ |u(x) -\tfrac{1}{(\nu \cdot \overline{a} \nu)^{1/2}} (x \cdot \nu)_+| \leq \delta r \ \hbox{ in } \ B_r \]
 for some $0 \leq \delta \leq \bar{\delta}$ then there is $\nu' \in S^{d-1}$ with $|\nu' - \nu| \leq C[\delta+\sigma^{\frac{1}{d+4}}+ r^{-\omega}]$ and
 \[ |u(x) -\tfrac{1}{(\nu' \cdot \overline{a} \nu')^{1/2}}(x\cdot \nu')_+ | \leq \eta^{1+\alpha}[\delta +C\sigma^{\frac{1}{d+4}}+C r^{-\omega}]r \ \hbox{ in } \ B_{\eta r}\]
 where $\omega>0$ is a dimensional constant.
\end{lemma}

Then \lref{eplip1step} can be iterated to obtain the following $C^{1,\alpha}$ iteration for $(R_0,\beta)$-almost minimizers.  Note that flatness improves down to a certain intermediate length scale, then gets worse until the iteration is forced to stop at scale $1$. 
\begin{lemma}[Flat implies $C^{1,\alpha}$ iteration for heterogeneous problem]\label{l.eplip}
 Suppose that $u : B_R \to [0,\infty)$ has $\|\grad u \|_{L^\infty(B_R)} \leq L$ and $u$ is an $(R_0,\beta)$-almost minimizer of $J$ in $B_R$. 
 
  There is $\bar{\delta}(d,\Lambda,L)>0$ and $c(d,\Lambda,L,\beta)>0$ so that if $0 < \delta \leq \bar{\delta}$ and
\[  |u(x) -\tfrac{1}{(\nu \cdot \overline{a} \nu)^{1/2}} (x \cdot \nu)_+| \leq \delta R \ \hbox{ in } \ B_R \ \hbox{ for some } \ 1 \leq R \leq cR_0\]
and $0 \in \partial \{u>0\}$ then for any $1\leq r \leq R$ there is a unit direction $\nu'$ so that
\[ |u(x) -\tfrac{1}{(\nu' \cdot \overline{a} \nu')^{1/2}}(x\cdot \nu')_+ | \leq C\left[\left(\frac{r}{R}\right)^\alpha(\delta + (R/R_0)^{\frac{\beta}{d+4}}) + r^{-\omega}\right] r \ \hbox{ for } \ x \in B_{r} \]
and
\[ |\nu' - \nu| \leq C\left[\delta + (R/R_0)^{\frac{\beta}{d+4}} + r^{-\omega}\right]\]
where $\omega>0$ is a dimensional constant.
\end{lemma}

This maintenance of flatness result down to the microscale directly implies that the free boundary stays a unit distance from a Lipschitz graph over the approximate normal direction $\nu$.

\begin{corollary}\label{c.lipbdry}
Under the hypotheses of \lref{eplip} there is a function $\psi : \{x \cdot \nu = 0\} \cap B_{R/8} \to \R$ with $\psi(0) = 0$ which is $2$-Lipschitz and
\[ \{x\cdot \nu \geq \psi(x')-C\} \cap B_{R/8} \subset \{u>0\} \cap B_{r/2} \subset \{x\cdot \nu \geq \psi(x')+C\} \cap B_{R/8}\]
where $x ' = x - (x \cdot \nu)\nu$, and $C = C(\Lambda,d,L)$.
\end{corollary}
Since this Corollary is just a technical reinterpretation of \lref{eplip} we will postpone the proof to \sref{lipbdry}.  

Now we proceed to the proofs of \lref{eplip1step} and \lref{eplip}.

\begin{proof}[Proof of \lref{eplip1step}]
Let $u$ as in the statement, and $\overline{u} \in u + H^1_0(\{u>0\} \cap B_r)$ be the upscaling of $u$ defined in \lref{energyconv} then,  by \lref{energyconv},
\[ J_0(\bar{u},B_r) \leq J_0(v,B_r) + C(L)[\sigma+r^{-\omega_0}]|B_r| \ \hbox{ for all $v \in u+H^1_0(B_{r})$}\]
where $\omega_0$ is the dimensional constant from \lref{energyconv}.  Also
\begin{equation}\label{e.baruests}
 \|\grad \bar{u}\|_{L^\infty(B_r)} \leq 2L \ \hbox{ and } \ \frac{1}{r}\|u - \bar{u}\|_{L^\infty(B_r)} \leq Lr^{-1/3}.
 \end{equation}
 Thus $\overline{u}$ has $\|\grad \bar{u}\|_{L^\infty(B_r)} \leq 2L$,
 \[ |\bar{u}(x) -\tfrac{1}{(\nu \cdot \overline{a} \nu)^{1/2}} (x \cdot \nu)_+| \leq [\delta+ Lr^{-1/3}]r  \ \hbox{ in } \ B_r,\]
 and satisfies 
 \[ J_0(\bar{u},B_r) \leq J_0(v,B_r)+\sigma_0|B_r| \ \hbox{ for all } \ v \in u + H^1_0(B_r)\] 
 with
 \[ \sigma_0 = C(L)[\sigma+r^{-\omega_0}].\]

 Applying \lref{dsalmostmin-iteration} we find, absorbing the $Cr^{-1/2}$ term into the $Cr^{-\frac{\omega_0}{(d+4)}}$ term,
\[| \overline{u}(x)- \tfrac{1}{(\nu' \cdot \overline{a} \nu')^{1/2}}(x\cdot \nu')_+ | \leq \eta^{1+\alpha} [\delta + C\sigma^{\frac{1}{d+4}} + Cr^{-\frac{\omega_0}{(d+4)}}] r \ \hbox{ in } \ B_{\eta r}(0)\]
for some $\nu'$ with
\[|\nu'- \nu| \leq C\left[ \delta + \sigma^{\frac{1}{d+4}} + r^{-\frac{\omega_0}{(d+4)}}\right].\]  
From here we will call $\omega = \frac{\omega_0}{d+4}$.  Then we carry the flatness condition back to $u$ using \eref{baruests}
\[ |u(x) - \tfrac{1}{(\nu' \cdot \overline{a} \nu')^{1/2}}x\cdot \nu'| \leq | \overline{u}(x)- \tfrac{1}{(\nu' \cdot \overline{a} \nu')^{1/2}}(x\cdot \nu')_+ | + C\eta^{-1}r^{-1/2}\eta r \ \hbox{ in } \ B_{\eta r}(0)\]
and the new error term can also be absorbed into the $Cr^{-\omega}$ term.

\end{proof}

\begin{proof}[Proof of \lref{eplip}]
Iterate \lref{eplip1step} taking $r_0 = R$, $\delta_0 = \delta$, 
\[ r_k = \eta r_{k-1}, \ \sigma_k = C(r_k/R_0)^\beta,\]
and
\[ \delta_k = \eta^{\alpha}(\delta_{k-1} + C\sigma_{k-1}^{\frac{1}{d+4}} + Cr_{k-1}^{-\omega}).\]
As long as
\[ \delta_k \leq \bar{\delta}\]
and $\sigma_0 \leq \bar{\sigma}$ then \lref{eplip} applies.  This holds, for example, if
\begin{equation}\label{e.requirements}
 \delta \leq \bar{\delta}/3, \ C(R/R_0)^{\frac{\beta}{d+4}} \leq \min\{\bar{\delta}/3,\bar{\sigma}\}, \ \hbox{ and } \ r_k \geq C\bar{\delta}^{-6(d+4)}
 \end{equation}
with $\bar{\delta}$ from \lref{eplip}.  Note that this requires that $R \leq c R_0$ for a small enough constant $c$ depending on $\beta$ and $\bar{\delta}$.   Call $\bar{k}$ to be the largest integer $k$ so that the third requirement of \eref{requirements} holds, note that $r_{\bar{k}} \leq C(d,\Lambda,L)$.

We compute the iteration of $\delta_k$
\begin{align*}
 \delta_k &= \eta^\alpha (\delta_{k-1} + C\sigma_{k-1}^{\frac{1}{d+4}} + Cr_{k-1}^{-\omega}) \\
 &= \eta^{k\alpha}\delta_0 + C\sum_{j=0}^{k-1} \eta^{(k-j)\alpha}[\sigma_{j}^{\frac{1}{d+4}} + r_{j}^{-\omega}] \\
 &\leq \eta^{k\alpha}\delta_0 + C\sum_{j=0}^{k-1} \eta^{(k-j)\alpha} [(r_j/R_0)^{\frac{\beta}{d+4}} + r_j^{-\omega}]\\
 &\leq \eta^{k\alpha}\delta_0 + C\sum_{j=0}^{k-1} \eta^{(k-j)\alpha} [\eta^{\frac{\beta}{d+4}j}(R/R_0)^{\frac{\beta}{d+4}} + \eta^{-{\omega} j}R^{-{\omega} }]\\
 & \leq \eta^{k\alpha} \delta_0 + C\eta^{k\alpha}(R/R_0)^{\frac{\beta}{d+4}} + C\eta^{ - k{\omega} }R^{-{\omega} } \\
 & \leq \left[ (r_k/R)^\alpha(\delta_0 + C(R/R_0)^{\frac{\beta}{d+4}}) + r_k^{-{\omega} }\right]
 \end{align*}
using for the second to last line that we can take $0 < \alpha < \frac{\beta}{d+4}$. We also obtain the following estimate of the variation of the normal directions
\[ |\nu_k- \nu_0| \leq \sum_{j=1}^k |\nu_j - \nu_{j-1}| \leq C \sum_{j=0}^{k-1} (\delta_{j}+\sigma_{j}^{\frac{1}{d+4}} + Cr_{j}^{-{\omega} }) \leq C[\delta_0 +(R/R_0)^{\frac{\beta}{d+4}}+ r_k^{-{\omega} }].\]
We are using several times here that we are summing geometric series and the finite sum is bounded by a constant times the largest term, of course the constant depends on $\eta$, which depends on $d,\Lambda,L$, and on $\beta$.

Finally the intermediate radii $r_{k} < r < r_{k-1}$ or $1 \leq r \leq r_{\bar{k}}$ are estimated by using the flatness on $B_{r_{k-1}}$ or $B_{r_{\bar{k}}}$ respectively.
\end{proof}
\begin{remark}\label{r.C1alpha-iteration-dini-hyp}
The almost minimality hypothesis in \lref{eplip} could be replaced by a more general one: for all $1 \leq r \leq R$ and any $v \in u+H^1_0(B_r(0))$
\begin{equation}\label{e.almostminimal-omeg2}
 J(u,B_r(0)) \leq J(v,B_r(0)) + \omega(r/R_0)|B_r|,
 \end{equation}
 as long as $\omega^{\frac{1}{d+4}}$ is a Dini modulus with $\omega(1) \leq 1$.  In the proof the parameter $\sigma_k$ would change to $\sigma_k = \omega(r_k/R_0) = \omega(\eta^kR/R_0)$.  This would then affect the estimate of the term
 \[ \sum_{j=0}^{k-1} \eta^{(k-j)\alpha} \omega(\eta^jR/R_0)^{\frac{1}{d+4}} \leq \eta^{k\alpha/2} + \sum_{k/2}^{k-1}\omega(\eta^jR/R_0)^{\frac{1}{d+4}} \leq \eta^{k\alpha/2}+C(\eta)\int_{0}^{\eta^{k/2}}\omega(t)^{\frac{1}{d+4}} \frac{dt}{t}.\]
This results in a decay of oscillations with an altered modulus
 \[ \tilde{\omega}(\frac{r}{R}) = \int_0^{(r/R)^{1/2}}\omega(t)^{\frac{1}{d+4}} \frac{dt}{t}\]
 assuming that this is not better than the $(r/R)^{\alpha/2}$ modulus.
\end{remark}
\subsection{Liouville property of flat energy minimizers on $\R^d$}\label{s.liouville} We conclude the section with an interesting corollary about the plane-like property of large scale flat {absolute minimizers} on the whole space (global). 

A function $u \in H^1_{loc}(\R^d)$ is called an absolute minimizer if it is minimizing with respect to compact perturbations, i.e. if
\[ J(u,B) \leq J(v,B) \ \hbox{ for all } \ v \in u + H^1_0(B)\]
where $B$ is \emph{any} finite radius ball in $\R^d$.  This property is also referred to as class-A minimizer in the literature \cite{CaffarellidelaLlave}. 

It is known that there are global plane-like absolute minimizers \cite{Valdinoci}, i.e. the minimizers are sandwiched between two translations of a planar solution of the effective problem $(\nu \cdot \overline{a}\nu)^{-1}(\nu \cdot x)_+$.  In general it seems to be a very difficult problem to classify all the global absolute minimizers. Even for homogeneous media there can be non planar global absolute minimizers in higher dimensions \cite{deSilva}, which is why the flatness assumption we have been using is essential. The following result says that all large scale flat global absolute minimizers grow sublinearly away from a planar solution of the effective problem.  It is not clear whether such solutions are a \emph{constant} distance from a planar solution of the effective problem, this seems like a very difficult question. 

\begin{theorem}\label{t.liouville}
Suppose that $u$ is an absolute minimizer of $J$ on $\R^d$ with
\[L = \limsup_{r \to \infty} \|\grad u\|_{\underline{L}^2(B_r)} < + \infty \]
and
 \[ \frac{1}{r}\inf_{\nu \in S^{d-1}}\sup_{B_r} |u(x) - \tfrac{1}{(\nu \cdot \bar{a} \nu)^{1/2}}(x \cdot \nu)_+| \leq \delta \]
  for all $r$ sufficiently large. 

There is $\bar{\delta}$ depending on universal constants and $L$ so that if $0 < \delta \leq \bar{\delta}$ then there is a slope $\nu_* \in S^{d-1}$ with $|\nu_* - \nu| \leq C\delta$ so that for all $r \geq 1$
\[ \frac{1}{r}\sup_{B_r} |u(x) - \tfrac{1}{(\nu_* \cdot \bar{a} \nu_*)^{1/2}}(x \cdot \nu_*)_+| \leq C r^{-{\omega} }.\]
\end{theorem}
\begin{remark}\label{r.c1alpha-difficulty}{\rm 
This kind of Liouville property is a standard corollary of the $C^{1,\alpha}$ estimate, see \cite{AKM,ArmstrongShen,AvellanedaLinHO,MoserStruwe} for related results.  In the case of standard elliptic homogenization the decay $Cr^{-\delta}$ can be upgraded to $Cr^{-1}$, the key point is subtracting off an appropriate corrector and then re-applying the estimates.  This relies on invariance of the equations with respect to vertical translations $u \mapsto u + c$, our PDE is more analogous to the minimization of integral functionals $\int a(x,u) \grad u \cdot \grad u \ dx$ where $a$ is periodic in $(x,u)$ and this invariance is missing.  See Moser \cite{Moser} and Moser and Struwe \cite{MoserStruwe} for some discussion of related Liouville questions for those models.  }
\end{remark}
\begin{proof}
By \tref{main1} $u$ is Lipschitz in $\R^n$ with constant $CL$ with $C$ universal.  Let $$\bar{\delta}(d,\Lambda,\|\grad u\|_{L^\infty})>0$$ from \lref{eplip}.  By \lref{eplip}, there is $\alpha <1$ so that for all $R$ sufficiently large (so that the $\bar{\delta}$ flatness assumption holds) and for $1 \leq r \leq R$
\[ \frac{1}{r}\inf_{\nu \in S^{d-1}} \sup_{B_r}|u(x) - \frac{1}{(\nu \cdot \bar{a} \nu)^{1/2}}(x \cdot \nu)_+| \leq  C\left((\frac{r}{R})^\alpha\left[\frac{1}{R}\inf_{\nu \in S^{d-1}} \sup_{B_R}|u(x) - \frac{1}{(\nu \cdot \bar{a} \nu)^{1/2}}(x \cdot \nu)_+|\right] + \frac{1}{r^{\omega}}\right)\]
First of all, since the term in square brackets on the right is smaller than $\delta$ for all $R$ sufficiently large, sending $R \to \infty$ we find, 
\[\frac{1}{r}\inf_{\nu \in S^{d-1}} \sup_{B_r}|u(x) - \frac{1}{(\nu \cdot \bar{a} \nu)^{1/2}}(x \cdot \nu)_+| \leq C r^{-\omega}. \]
For each $r \geq 1$ let $\nu(r) \in S^{d-1}$ achieve the previous infimum.  Then this estimate shows that for all $r \leq s \leq 2r$
\[ \sup_{B_r}|\frac{1}{(\nu(s) \cdot \bar{a} \nu(s))^{1/2}}(x \cdot \nu(s))_+ - \frac{1}{(\nu(r) \cdot \bar{a} \nu(r))^{1/2}}(x \cdot\nu(r))_+| \leq C r^{1-\omega}\]
applying this with $x =  \frac{1}{\Lambda}r (\nu(s) \cdot \bar{a} \nu(s))^{1/2}\nu(r)$ and with $x = \frac{1}{\Lambda}r (\nu(r) \cdot \bar{a} \nu(r))^{1/2}\nu(s)$ gives
\[ |\nu(r) \cdot \nu(s) - \frac{(\nu(s) \cdot \bar{a} \nu(s))^{1/2}}{(\nu(r) \cdot \bar{a} \nu(r))^{1/2}}| + |\nu(r) \cdot \nu(s) - \frac{(\nu(r) \cdot \bar{a} \nu(r))^{1/2}}{(\nu(s) \cdot \bar{a} \nu(s))^{1/2}}| \leq Cr^{-\omega}\]
since $\nu(r) \cdot \nu(s)$ is close to two numbers which are reciprocals of one another
\[ \frac{1}{2}|\nu(r) - \nu(s)|^2 = 1-\nu(r) \cdot \nu(s) \leq Cr^{-\omega}.\]
Thus the sequence $\nu(2^kr_0)$ converges as $k \to \infty$, call $\nu_*$ to be the limit. Now for any $r \geq r_0$ choose $k$ so that $2^{k-1}r_0 < r \leq 2^k r$ and then, using the previous estimate again,
\[ |\nu(r) - \nu_*| \leq |\nu(r) - \nu(2^kr_0)| + \sum_{j=k}^\infty |\nu(2^{j+1}r_0) - \nu(2^jr_0)| \leq C r^{-\omega}.\]
\end{proof}

\appendix

\section{Proofs of technical results}
In this section we provide the detailed proofs of several of the technical results appearing in the paper.

\subsection{Proof of \lref{poissonkernel}}\label{s.poissonkernel}

The result is a standard Hopf Lemma barrier argument, using that $a$ is Lipschitz to write the equation in non-divergence form. We just need to establish that the Green's function $G(x,y)$ for the operator $- \grad\cdot(a(rx) \grad \cdot)$ in $B_1$ has
\[ \inf_{y\in B_{\eta}} G(0,y) \geq c\]
for $\eta$ and $c$ depending on universal constants and $r_0$.  

For this take $G_0(x,y)$ to be the Green's function for $- \grad \cdot (a(0) \grad \cdot)$ in $B_1$ and $G(x,y)$ to be the Green's function for $- \grad\cdot(a(rx) \grad \cdot)$ in $B_1$.  Call $w(x) = G(x,0) - G_0(x,0)$ so that
\[ - \grad \cdot(a(0) \grad w) = \grad \cdot ((a(rx)-a(0)) \grad G(x,0))  \ \hbox{ and } \ w = 0 \ \hbox{ on } \ \partial B_1.\]
Then
\[ w(x) = \int_{B_r} \grad G_0(x,y) (a(y) - a(0))\grad G(y,0) \ dy. \]
Applying Green's function bounds \cite{MasmoudiGerardVaret}[Lemma 2.5],
\[ |\grad G(x,y)| \leq C|x-y|^{1-n}\]
 which hold uniformly in $r>0$, and the H\"older continuity of $a$
\begin{align*}
|w(x)| \leq C \int_{B_1} |x-y|^{1-n} r^{\gamma}|y|^\gamma |y|^{1-n} \ dy.
\end{align*}
We estimate in a standard way splitting the integral
\begin{align*}
|w(x)| &\leq C r^\gamma \bigg[\int_{1 \geq |y| \geq 2|x|} |y|^{\gamma -2(n-1)} \ dy + \int_{|y| \leq |x|/2} |x|^{1-n} |y|^{\gamma +1-n} \ dy\cdots\\
&\quad \quad \quad \quad  \cdots+\int_{ |x|/2 \leq |y| \leq 2|x|} |x-y|^{1-n} |x|^{\gamma +1-n} \ dy\bigg]
\end{align*}
this results in
\[ |w(x)| \leq Cr^\gamma |x|^{2-n+\gamma}.\]
Because $G_0$ solves a constant coefficient equation we have $G_0(x,0) \geq c|x|^{2-n}$ so
\[ G(x,0) \geq G_0(x,0) - w(x) \geq (c-Cr^\gamma|x|^\gamma)|x|^{2-n}\]
so we can guarantee a universal positive lower bound on $|x| \leq cr_0^{-1}$.
\qed

\subsection{Proof of \lref{domainbdryenergy}}\label{s.domainbdryenergy}

Without loss, by a linear transformation, we can assume that $\overline{a} = id$.  This comes at the cost of changing the coefficient of ${\bf 1}_{\{u>0\}}$ to $\textup{det}(\overline{a})$, since this is just another constant depending on ellipticity we will just pretend it is $1$ to reduce the number of constants below.

Step 1 (H\"older continuity):  We consider continuity at a boundary point $y \in \partial B_1$.  Call $L = \max\{1,\|\grad g\|_\infty\}$. First we use a barrier argument to show that $u$ is positive in a ball of radius $r = c(d)L^{-1}g(y)$ centered at $y$.  Then $u$ is harmonic in $B_1 \cap B_r(y)$.  

Since $u$ is subharmonic in the entire $B_1$, see \cite{AltCaffarelli}, we can do a standard barrier argument from above to find, for any $0 < \alpha < 1$,
\[ \max_{B_1 \cap B_\rho(y)}u \leq g(y) + C(d,\alpha)L \rho^\alpha.   \]
This argument also works from below if $r \geq 2$ since then $u$ is harmonic in the entire ball $B_1$.

 In the case $r \leq 2$ we can apply a similar barrier argument from below in $B_r(y) \cap B_1$ to find for all $0 < \alpha < 1$ and $0 \leq \rho \leq r$
\[ \min_{B_1 \cap B_\rho(y)} u \geq g(y) - C(d,\alpha)g(y)(\rho/r)^\alpha = g(y) - CL^{\alpha}g(y)^{1-\alpha}\rho^{\alpha}.  \]
Note that since $r \leq 2$ also $L^{\alpha}g(y)^{1-\alpha} \leq C L$ so the estimate from below is of the same form as the one from above.

We just need to check the positivity in $B_r(y)$ for $r = c_0 L^{-1}g(y)$ and $c_0>0$ chosen sufficiently small dimensional.   We can assume to start that $c_0>0$ small enough so that
\[ \inf_{B_{2r}(y)} g \geq g(y) - 2L r \geq \frac{1}{2}g(y).\]
 Let $\phi$ be the harmonic function in $B_{2} \setminus B_{1}$ which is equal to $1$ on $\partial B_{1}$ and $0$ on $\partial B_2$, extend $\phi$ by $0$ outside of $B_2$.  Make a family of barriers with $r = c_0L^{-1}g(y)$
\[ \psi_{t}(x) = \frac{1}{2}g(y)  \phi(r^{-1} (x-z_t))\]
where $(z_t)_{t \in [0,1]}$ is 
\[z_t = (1+\frac{2r}{(1+t)})y \]
so that $B_{2r}(z_0)$ is exterior tangent to $B_1$ at $y$ and $B_{r}(z_1)$ is exterior tangent to $B_1$ at $y$.  Note that 
\[ \psi_t \leq \frac{1}{2}g(y) \leq g(x) \ \hbox{ for $x$ in $B_{2r}(z_t) \cap B_1 \subset B_{2r}(y) \cap B_1$}\]
and $\psi_t = 0 \leq g$ outside of $B_{2r}(z_t)$.

The plan is to show that $\psi_t$ are subsolutions of the Euler-Lagrange equation and below $g$ on $\partial B_1$ and then slide inwards from $t=0$ to $t=1$ to find $\psi_1 \leq u$.  Note that $\psi_0 \leq u$ in $B_1$ because $\psi_0 \equiv 0$ in $B_1$.  We have also established above that $\psi_t \leq g$ on $\partial B_1$ for all $t \in [0,1]$.  

Note that on $\partial B_{2r}(z_t)$
\[ |\grad \psi_t| = c(d)c_0^{-1}L \geq 1\]
where the last inequality comes from $L\geq 1$ and choosing $c_0$ large enough dimensional.  In particular $\psi_t$ is a smooth strict subsolution of the Euler-Lagrange equation for $J_0$ outside of $B_{r}(z_t)$.  

Technically we are using that $u$ being a $J_0$ minimizer in $B_1$ is a viscosity solution of the Euler-Lagrange equation
\[ - \Delta u = 0 \ \hbox{ in } \ \{u>0\} \cap B_1\ \hbox{ and } \ |\grad u| = 1 \ \hbox{ on } \ \partial \{u>0\} \cap B_1.\]
Then we are applying a sliding maximum principle, for reference see \cite{CaffarelliSalsa}[Theorem 2.2].

Step 2 (Energy comparison): Define an energy competitor
\[v = \phi_\delta g + u (1-\phi_\delta) \]
where $\phi_\delta$ is a Lipschitz function with $0 \leq \phi_\delta \leq 1$, $|\grad \phi_\delta| \leq C(\delta r)^{-1}$, and
\[ \phi_\delta(x) = \begin{cases} 1 & |x| \leq (1-2\delta)r \\
0 & |x| \geq  (1-\delta)r.
\end{cases}\]
  Since $\phi_\delta \equiv 0$ in a neighborhood of $\partial B_r$ we do have $v \in H^1_0(U) + g$.  

Energy comparison gives
\[ J(u,B_r) \leq J(v,B_r) = J(u,B_{(1-2\delta)r}) + J(v,B_r \setminus B_{(1-2\delta)r})\]
so
\begin{align*}
 J(u,B_r \setminus B_{(1-\delta)r}) &\leq J(v,B_r \setminus B_{(1-2\delta)r}) - J(u,B_{(1-\delta)r} \setminus B_{(1-2\delta)r})\\
 &\leq \int_{B_r \setminus B_{(1-2\delta)r}} |\phi_\delta \grad g + (1-\phi_\delta) \grad u + (g - u)\grad \phi_\delta|^2 \ dx \\
 & \quad \quad \cdots - \int_{B_{(1-\delta)r} \setminus B_{(1-2\delta)r}}|\grad u|^2dx + \Lambda|B_r \setminus B_{(1-2\delta)r}| \\
  &\leq \int_{B_r \setminus B_{(1-2\delta)r}} (1+\eta)|\phi_\delta \grad g + (1-\phi_\delta) \grad u|^2 + (1+\frac{1}{\eta})|(g - u)\grad \phi_\delta|^2 \ dx \\ & \quad \quad \cdots- \int_{B_{(1-\delta)r} \setminus B_{(1-2\delta)r}}|\grad u|^2dx + \Lambda|B_r \setminus B_{(1-2\delta)r}| \\
    &\leq \int_{B_r \setminus B_{(1-2\delta)r}} (1+\eta)\phi_\delta |\grad g|^2 + (1+\eta)(1-\phi_\delta) |\grad u|^2 + (1+\frac{1}{\eta})|(g - u)|^2|\grad \phi_\delta|^2 \ dx \\
    & \quad \quad \cdots- \int_{B_{(1-\delta)r} \setminus B_{(1-2\delta)r}}|\grad u|^2dx + \Lambda|B_r \setminus B_{(1-2\delta)r}|\\
        &\leq (\Lambda+2L^2)|B_r \setminus B_{(1-2\delta)r}| + \int_{B_r \setminus B_{(1-2\delta)r}}\eta  |\grad u|^2 + (1+\frac{1}{\eta})|(g - u)|^2|\grad \phi_\delta|^2 \ dx   \\
        & \leq \eta  \|\grad u\|^2_{L^2(B_r)} + C(1+\eta^{-1} \delta^{2(\alpha-1)}r^{-2})|B_r \setminus B_{(1-2\delta)r}|.
 \end{align*}
 We used Young's inequality in the form $2 ab \leq \eta a^2 + \frac{1}{\eta} b^2$ for the third inequality, and we used convexity of $a \mapsto |a|^2$ for the fourth inequality. Note that 
 \[\|\grad u\|^2_{L^2(B_r)} \leq J(u,B_r) \leq J(g,B_r) \leq C(L)|B_r|\]
 and $|B_r \setminus B_{(1-2\delta)r}| \leq C \delta |B_r|$ so
 \[ J(u,B_r \setminus B_{(1-\delta)r}) \leq C(\eta +\delta+\eta^{-1} \delta^{2(\alpha-1)+1}r^{-2}) |B_r|\]
 choosing $\eta = \delta^{\frac{1}{2} + (\alpha - 1)}r^{-1}$ to match the first and third terms
 \[ J(u,B_r \setminus B_{(1-\delta)r}) \leq C(\delta + r^{-1}\delta^{\alpha - \frac{1}{2}}) |B_r|.\]

\subsection{Proof of \cref{lipbdry}}\label{s.lipbdry}

We can take the flatness direction in the assumption to be $\nu = e_d$.  By the assumption of the corollary then
\[ |u(x) -\tfrac{1}{e_d \cdot \overline{a} e_d} (x_d)_+| \leq \delta r.\]
Let $y = (y',y_d)$ and $z = (z',z_d)$ be two points in $\partial \{u>0\} \cap B_{r/8}$.  So, in particular, 
\[ z \in B_{r/4}(y) \subset B_{r/2}(y) \subset B_{5r/8}(0)\]  
Now we aim to apply \lref{eplip} to $u$ in $B_{r/2}(y)$.  By the original flatness assumption $u$ is $2 \delta r$-flat in $B_{r/2}(y)$ so as long as $4 \delta \leq \bar{\delta}$ from \lref{eplip} we can apply the result at the scale $s = |y-z| + s_0$ to obtain
\[ |u(x) -\tfrac{1}{(\nu \cdot \overline{a} \nu)^{1/2}} (x \cdot \nu)_+| \leq A r \ \hbox{ with } \ A = C(\bar{\delta} + s^{-\omega})\]
 for a direction $\nu$ with
\[ |\nu - e_d| \leq A\]
where the scale $s_0$ is chosen depending on $L$ and universal constants so that we can have $A \leq \frac{1}{4}$, also choosing $\bar{\delta}$ smaller if necessary.

In particular
\[ |(z-y) \cdot \nu| \leq A s = A (|z-y| + s_0)\]
and
\[ |(z-y) \cdot \nu| \geq |(z-y) \cdot e_d| - |(z-y) \cdot(e_d - \nu)| \geq |x_d - z_d| - A|z-y|.\]
So
\[ |z_d - y_d| \leq 2A(|z-y|+s_0) \leq  \frac{1}{2}(|z_d-y_d| + |z'-y'|+s_0) \]
and we can move the first term over to the left and get
\begin{equation}\label{e.zdyd}
 |z_d - y_d| \leq 2|z'-y'| +2s_0.
 \end{equation}
The final conclusion follows then from taking the graph $\psi$ to be the supremum of cones $a - 2|x'-x_0'|$ whose subgraph lies inside $\{u>0\}$ in $B_{r/8}$.  Then \eref{zdyd} implies that $\partial \{u>0\}$ is contained between the graphs of $\psi$ and $\psi + 2s_0$.

\qed

\subsection{Density estimates}\label{s.density} In this appendix we show interior and exterior density estimates.

The main result of the section is
\begin{proposition}\label{p.densityestimates}
Suppose that $u$ is non-negative $H^1$ function in a ball $B_{r}(0)$ with
\begin{enumerate}[label = (\alph*)]
\item $u$ is $L$-Lipschitz
\item $u$ is $\ell$-non-degenerate at its free boundary
\item There is $\sigma>0$ so that for any ball $B_\rho \subset B_r(0)$
\[ J(u,B_\rho) \leq J(v,B_\rho)+\sigma |B_\rho| \]
where $v \in u + H^1_0(B_\rho)$ is the $a$-harmonic replacement.
\end{enumerate}
There is $\sigma_0>0$ and $c>0$, $C \geq 1$ depending on $L$, $\ell$ and universal parameters so that if $\sigma \leq \sigma_0$ then
\begin{equation}\label{e.iedensity} 
 c\leq \frac{|\{u>0\} \cap B_\rho(x)|}{|B_\rho|} \leq 1-c \ \hbox{ for any } \ B_\rho \subset B_r(0).
\end{equation}
\end{proposition}

 The interior density only requires the Lipschitz and non-degenerate hypotheses so the conclusion is standard, the exterior density estimate does require the almost minimality property and some inputs from homogenization.

\begin{proof}[Proof of \pref{densityestimates}]
The inner density follows directly from non-degeneracy and the Lipschitz estimate, see for example \cite{CaffarelliSalsa}.

Now we consider the exterior density.  We follow a standard line of argument, see \cite[Lemma 5.1]{VelichkovNotes}, but we need to additionally use homogenized problem replacement at large scales and continuity of $a$ at small scales to apply the constant coefficient argument (Poisson kernel lower bound).  

Let $h$ be the $a$-harmonic replacement of $u$ in $B_r$.  Then
\begin{align*}
 \sigma |B_r|+\Lambda|\{u = 0\} \cap B_r| &\geq \int_{B_r} a(x) \grad u \cdot \grad u - a(x) \grad h \cdot \grad h \ dx \\
 &= \int_{B_r} a(x) \grad (u-h) \cdot \grad (u-h) \ dx \\
 &\geq \Lambda^{-1}\int_{B_r} |\grad (u-h)|^2 \ dx.
 \end{align*}
We will first argue for $r \geq r_0 \geq 1$ large enough but bounded by universal constants, then we will argue for $r \leq r_0$ at the end.  Let $\overline{h}$ be the $\overline{a}$-harmonic replacement of $h$ in $B_r$ then by \tref{quanthomregulardomain}
\[ r^{-2}\|h-\overline{h}\|^2_{L^2(B_r)} \leq Cr^{-\alpha}|B_r|. \]
By Poincar\'e inequality in $B_r$
\begin{align*}
\int_{B_r} |\grad (u-h)|^2 \ dx 
 &\geq \frac{\Lambda^{-1} c_d}{r^2} \int_{B_r} (h-u)^2 \ dx \\
 & \geq \frac{c}{r^2} \int_{B_r} (\overline{h} -u)^2 \ dx - \frac{c}{r^2} \int_{B_r} (\overline{h} -h)^2 \ dx\\
 & \geq \frac{c}{|B_r|}\left(\frac{1}{r}\int_{B_r} h-u \ dx\right)^2 - C r^{-\alpha}|B_r| .
 \end{align*} 
 By the non-degeneracy assumption of $u$ we know $\max_{\partial B_r} u \geq \ell r$.  By Lipschitz continuity $u \geq \frac{1}{2}\ell r$ in a $\frac{1}{2}\frac{\ell}{L}r$-neighborhood of the point where the maximum is achieved so
 \[ \int_{\partial B_r} u(x) dS \geq \frac{1}{2}\ell r\int_{\partial B_r} {\bf 1}_{\{u(x) \geq \frac{1}{2}\ell r\}} dS \geq c \tfrac{\ell^2}{L}r |\partial B_r| \] 
 By the Poisson kernel formula for $\overline{h}$, now dropping the $\ell$ and $L$ dependence of the constant,
 \[ \overline{h}(0) \geq c r. \]
and by Harnack inequality (recall $\overline{h}$ is harmonic up to a linear transformation)
\[ \overline{h}(x) \geq c \overline{h}(0) \geq c r \ \hbox{ in } \ x\in B_{r/2}.\]
The Lipschitz estimate of $u$ gives
 \[ u(x) \leq L|x| .\]
 Thus
 \[ u(x) - \overline{h}(x) \geq \frac{1}{2}c \ell r \ \hbox{ in } |x| \leq c\ell r/2L.\]
 Combining this with previous
 \begin{align*}
 \frac{1}{|B_r|}\left(\frac{1}{r}\int_{B_r} \overline{h}-u \ dx\right)^2 \geq \frac{1}{|B_r|}\left(\frac{1}{r}\int_{B_{c\ell r/2L}} \frac{1}{2}c \ell r  \ dx\right)^2 \geq 3c_0 |B_r|.
 \end{align*}
 For $r \geq r_0$ (universal) the term $C \mathcal{E}(\frac{1}{r}) \leq c_0$ so
 \[ \int_{B_r} |\grad (u-h)|^2 \ dx  \geq 2 c_0 |B_r|\]
 and for $\sigma \leq c_0$ as well we get finally
 \[ \Lambda|\{u = 0\} \cap B_r| \geq c_0 |B_r|.\]
 
 Next we argue for $r \leq r_0$.  Actually it is more convenient to argue for $r \leq r_1 \ll 1$ again a universal constant.  For intermediate values $r_1 \leq r \leq r_0$ we will just use a lower bound 
 \[ |\{u=0\} \cap B_r| \geq |\{u=0\} \cap B_{r_1}| \geq c_1|B_{r_1}| \geq c_1 (\frac{r_1}{r_0})^d|B_r|.\]
 The small $r_1$ case follows a very similar line of argument to the previous using the $a(0)$-harmonic replacement $\tilde{h}$ in $B_r$ (instead of $\overline{a}$-harmonic replacement) along with the estimate
 \[ \|h - \tilde{h}\|_{\underline{L}^2(B_r)}^2 \leq Cr^2\|\grad(h-\tilde{h})\|_{\underline{L}^2(B_r)}^2\leq Cr^2r^{2\gamma} |B_r|\]
 which follows from the standard energy estimate
 \[ \|\grad(h-\tilde{h})\|_{L^2(B_r)} \leq C\|a - a(0)\|_{L^\infty(B_r)}\|\grad \tilde{h}\|_{L^2(B_r)}\leq Cr^{\gamma}(\int_{\partial B_r}|u||\grad u| )^{1/2} \leq Cr^\gamma L|B_r|^{1/2}\]
 using that $a$ is $C^{0,\gamma}$.
 \end{proof}


  \bibliographystyle{plain}
\bibliography{fbhomreg-articles}

\begin{thebibliography}{10}

\bibitem{AleksanyanKuusi}
Gohar Aleksanyan and Tuomo Kuusi.
\newblock Quantitative homogenization for the obstacle problem and its free
  boundary.
\newblock {\em arXiv:2112.10879 [math.AP]}, 2021.

\bibitem{AltCaffarelli}
H.~W. Alt and L.~A. Caffarelli.
\newblock Existence and regularity for a minimum problem with free boundary.
\newblock {\em J. Reine Angew. Math.}, 325:105--144, 1981.

\bibitem{ApushkinskayaNazarov}
Darya~E. Apushkinskaya and Alexander~I. Nazarov.
\newblock On the boundary point principle for divergence-type equations.
\newblock {\em Atti Accad. Naz. Lincei Rend. Lincei Mat. Appl.},
  30(4):677--699, 2019.

\bibitem{AKM}
Scott Armstrong, Tuomo Kuusi, and Jean-Christophe Mourrat.
\newblock {\em Quantitative stochastic homogenization and large-scale
  regularity}, volume 352 of {\em Grundlehren der mathematischen Wissenschaften
  [Fundamental Principles of Mathematical Sciences]}.
\newblock Springer, Cham, 2019.

\bibitem{ArmstrongShen}
Scott~N. Armstrong and Zhongwei Shen.
\newblock Lipschitz estimates in almost-periodic homogenization.
\newblock {\em Comm. Pure Appl. Math.}, 69(10):1882--1923, 2016.

\bibitem{ArmstrongSmart}
Scott~N. Armstrong and Charles~K. Smart.
\newblock Quantitative stochastic homogenization of convex integral
  functionals.
\newblock {\em Ann. Sci. \'{E}c. Norm. Sup\'{e}r. (4)}, 49(2):423--481, 2016.

\bibitem{AvellanedaLin}
Marco Avellaneda and Fang-Hua Lin.
\newblock Compactness methods in the theory of homogenization.
\newblock {\em Comm. Pure Appl. Math.}, 40(6):803--847, 1987.

\bibitem{AvellanedaLinHO}
Marco Avellaneda and Fang-Hua Lin.
\newblock Un th\'{e}or\`eme de {L}iouville pour des \'{e}quations elliptiques
  \`a coefficients p\'{e}riodiques.
\newblock {\em C. R. Acad. Sci. Paris S\'{e}r. I Math.}, 309(5):245--250, 1989.

\bibitem{CaffarelliLee}
Luis Caffarelli and Ki~Ahm Lee.
\newblock Homogenization of oscillating free boundaries: the elliptic case.
\newblock {\em Comm. Partial Differential Equations}, 32(1-3):149--162, 2007.

\bibitem{CaffarelliSalsa}
Luis Caffarelli and Sandro Salsa.
\newblock {\em A geometric approach to free boundary problems}, volume~68 of
  {\em Graduate Studies in Mathematics}.
\newblock American Mathematical Society, Providence, RI, 2005.

\bibitem{CaffarelliIII}
Luis~A. Caffarelli.
\newblock A {H}arnack inequality approach to the regularity of free boundaries.
  {III}. {E}xistence theory, compactness, and dependence on {$X$}.
\newblock {\em Ann. Scuola Norm. Sup. Pisa Cl. Sci. (4)}, 15(4):583--602
  (1989), 1988.

\bibitem{CaffarellidelaLlave}
Luis~A. Caffarelli and Rafael de~la Llave.
\newblock Planelike minimizers in periodic media.
\newblock {\em Communications on Pure and Applied Mathematics},
  54(12):1403--1441, 2001.

\bibitem{DavidEngelsteinSVGToro}
Guy David, Max Engelstein, Mariana Smit Vega~Garcia, and Tatiana Toro.
\newblock Regularity for almost-minimizers of variable coefficient
  {B}ernoulli-type functionals.
\newblock {\em Math. Z.}, 299(3-4):2131--2169, 2021.

\bibitem{DavidEngelsteinToro}
Guy David, Max Engelstein, and Tatiana Toro.
\newblock Free boundary regularity for almost-minimizers.
\newblock {\em Adv. Math.}, 350:1109--1192, 2019.

\bibitem{DavidToro}
Guy David and Tatiana Toro.
\newblock Regularity of almost minimizers with free boundary.
\newblock {\em Calc. Var. Partial Differential Equations}, 54(1):455--524,
  2015.

\bibitem{deSilva}
Daniela De~Silva.
\newblock Free boundary regularity for a problem with right hand side.
\newblock {\em Interfaces Free Bound}, 13(2):223--238, 2011.

\bibitem{DeSilvaSavin}
Daniela De~Silva and Ovidiu Savin.
\newblock Almost minimizers of the one-phase free boundary problem.
\newblock {\em Comm. Partial Differential Equations}, 45(8):913--930, 2020.

\bibitem{Feldman}
William~M. Feldman.
\newblock Limit shapes of local minimizers for the alt-{C}affarelli energy
  functional in inhomogeneous media.
\newblock {\em Arch. Ration. Mech. Anal.}, 240(3):1255--1322, 2021.

\bibitem{FeldmanShape}
William~M. Feldman.
\newblock Quantitative homogenization of principal dirichlet eigenvalue shape
  optimizers.
\newblock {\em arXiv:2209.01446 [math.AP]}, 2022.

\bibitem{FeldmanSmart}
William~M. Feldman and Charles~K. Smart.
\newblock A free boundary problem with facets.
\newblock {\em Arch. Ration. Mech. Anal.}, 232(1):389--435, 2019.

\bibitem{MasmoudiGerardVaret}
David G\'{e}rard-Varet and Nader Masmoudi.
\newblock Homogenization and boundary layers.
\newblock {\em Acta Math.}, 209(1):133--178, 2012.

\bibitem{JerisonSavin}
David Jerison and Ovidiu Savin.
\newblock Some remarks on stability of cones for the one-phase free boundary
  problem.
\newblock {\em Geom. Funct. Anal.}, 25(4):1240--1257, 2015.

\bibitem{Kim}
Inwon~C. Kim.
\newblock Homogenization of a model problem on contact angle dynamics.
\newblock {\em Comm. Partial Differential Equations}, 33(7-9):1235--1271, 2008.

\bibitem{Moser}
J\"{u}rgen Moser.
\newblock Minimal solutions of variational problems on a torus.
\newblock {\em Ann. Inst. H. Poincar\'{e} Anal. Non Lin\'{e}aire},
  3(3):229--272, 1986.

\bibitem{MoserStruwe}
J\"{u}rgen Moser and Michael Struwe.
\newblock On a {L}iouville-type theorem for linear and nonlinear elliptic
  differential equations on a torus.
\newblock {\em Bol. Soc. Brasil. Mat. (N.S.)}, 23(1-2):1--20, 1992.

\bibitem{Shen}
Zhongwei Shen.
\newblock {\em Periodic homogenization of elliptic systems}, volume 269 of {\em
  Operator Theory: Advances and Applications}.
\newblock Birkh\"{a}user/Springer, Cham, 2018.
\newblock Advances in Partial Differential Equations (Basel).

\bibitem{Trey}
Baptiste Trey.
\newblock Lipschitz continuity of the eigenfunctions on optimal sets for
  functionals with variable coefficients.
\newblock {\em ESAIM: COCV}, 26:89, 2020.

\bibitem{Valdinoci}
Enrico Valdinoci.
\newblock Plane-like minimizers in periodic media: jet flows and
  {G}inzburg-{L}andau-type functionals.
\newblock {\em J. Reine Angew. Math.}, 574:147--185, 2004.

\bibitem{VelichkovNotes}
Bozhidar Velichkov.
\newblock {\em Regularity of the One-phase Free Boundaries}.
\newblock Lecture Notes of the Unione Matematica Italiana. Springer, Cham,
  2023.

\end{thebibliography}

\end{document}